\documentclass[reqno,12pt,a4paper]{amsart}
\usepackage{booktabs,amssymb,amsmath,amsthm,appendix,graphicx}
\usepackage[bookmarksnumbered]{hyperref}
\usepackage[usenames,dvipsnames,svgnames,table]{xcolor}
\usepackage[ascii]{inputenc}
\usepackage[noend]{algpseudocode}

\newtheorem{theorem}{Theorem}[section]
\newtheorem{proposition}[theorem]{Proposition}
\newtheorem{lemma}[theorem]{Lemma}
\newtheorem{definition}[theorem]{Definition}

\newtheorem{remark}[theorem]{Remark}
\newtheorem{example}[theorem]{Example}
\numberwithin{equation}{section}

\newcommand\fractional[1]{\left\{#1\right\}}

\newcommand\floor[1]{\lfloor{#1}\rfloor}

\newcommand{\uu}{{\mathfrak{u}}}
\newcommand{\lp}{{\mathfrak{l}}}
\newcommand{\C}{{\mathbb C}}

\newcommand{\R}{{\mathbb R}}

\newcommand{\Z}{{\mathbb Z}}

\newcommand{\CA}{{\mathcal A}}

\newcommand{\CF}{{\mathcal F}}

\newcommand{\CH}{{\mathcal H}}

\newcommand{\CN}{{\mathcal N}}

\newcommand{\CP}{{\mathcal P}}

\newcommand{\CW}{{\mathcal W}}

\newcommand{\Cone}{\operatorname{Cone}}

\newcommand{\Id}{\operatorname{Id}}

\renewcommand{\mod}{\operatorname{mod}}

\newcommand{\Sym}{\operatorname{Sym}}

\newcommand{\Res}{\operatorname{Res}}
\newcommand{\SL}{\operatorname{SL}}

\newcommand{\SU}{\operatorname{SU}}

\newcommand{\U}{\operatorname{U}}

\renewcommand{\a}{{\mathfrak{a}}}

\renewcommand{\c}{{\mathfrak{c}}}

\newcommand{\g}{{\mathfrak{g}}}
\renewcommand{\k}{{\mathfrak{k}}}
\renewcommand{\l}{{\mathfrak{l}}}

\newcommand{\z}{{\mathfrak{z}}}

\renewcommand{\t}{{\mathfrak{t}}}
\renewcommand{\u}{{\mathfrak{u}}}

\newcommand{\la}{\langle}
\newcommand{\ra}{\rangle}

\allowdisplaybreaks[4]
\makeatletter
\let\@wraptoccontribs\wraptoccontribs
\makeatother
\hypersetup{pdftitle={Computation of dilated Kronecker coefficients}}

\begin{document}
\title{Computation of dilated Kronecker coefficients}
\renewcommand\rightmark{\MakeUppercase{Computation of dilated Kronecker coefficients}}
\author{V. Baldoni}
\address{Velleda Baldoni: Dipartimento di Matematica, Universit\`a degli studi di Roma ``Tor Vergata'', Via della ricerca scientifica 1, I-00133, Italy}
\email{baldoni@mat.uniroma2.it}
\author{M. Vergne}
\address{Mich\`ele Vergne: Universit\'e Paris 7 Diderot, Institut Math\'ematique de
Jussieu, Sophie Germain, case 75205, Paris Cedex 13}
\email{vergne@math.jussieu.fr}
\contrib[with an appendix by]{Michael Walter}
\address{Michael Walter: Stanford Institute for Theoretical Physics, Stanford University, Stanford, California 94305, USA}
\email{michael.walter@stanford.edu}
\date{}


\begin{abstract}
The computation of Kronecker coefficients is a challenging problem with a variety of applications.
In this paper we present an approach based on methods from symplectic geometry and residue calculus.
We outline a general algorithm for the problem and then illustrate its effectiveness in several interesting examples.
Significantly, our algorithm does not only compute individual Kronecker coefficients, but also symbolic formulas that are valid on an entire polyhedral chamber.
As a byproduct, we are able to compute several Hilbert series.
\end{abstract}

\maketitle
\tableofcontents

\section{Introduction}

Let $N=n_1n_2\cdots n_s$ where $n_1,\, n_2,\ldots ,n_s$ are positive integers, and write $\C^N=\C^{n_1} \otimes \C^{n_2} \otimes \cdots  \otimes \C^{n_s} $.
Endow each $\C^{n_i}$ with the usual Hermitian inner product.
Then the group $U(n_1)\times \cdots \times U(n_s)$ acts unitarily on the complex vector space $\C^{N}$ via the exterior tensor product, i.e.\ $(k_1,\ldots ,k_s)(v_1\otimes \cdots\otimes v_s)=k_1v_1\otimes \cdots \otimes k_sv_s.$
By means of this action we obtain a representation of  $U(n_1)\times \cdots \times U(n_s)$ on $\Sym(\C^N)$, the full symmetric algebra of $\C^{N}$.
The aim of this article is to give an algorithm to compute the multiplicity of an irreducible representation of $U(n_1)\times \cdots \times U(n_s)$ in $\Sym(\C^N)$.
As a corollary, we obtain an algorithm to compute the corresponding Hilbert series.

The algebra $\Sym(\C^N)$ has a natural grading by degree.
For notational convenience we let $\Sym^c(\C^N)$ denote the space of symmetric tensors of degree $c$, so we have $\Sym(\C^N)=\bigoplus_{c=0}^{\infty}\Sym^c(\C^N).$
The action of $U(n_1)\times \cdots \times U(n_s)$ provides a refinement of this decomposition, namely,
\begin{equation}\label{eq:kron sym}
  \Sym(\C^N) = \bigoplus_{\nu_1,\dots,\nu_s} g(\nu_1,\nu_2,\ldots, \nu_s) \, V^{U(n_1)}_{\nu_1}\otimes \cdots \otimes V^{U(n_s)}_{\nu_s},
\end{equation}
where the $V_{\nu_j}^{U(n_j)}$ are polynomial representations of $U(n_j)$ indexed by their highest weights $\nu_j$, which we can identify with Young diagrams that have no more than $n_j$ rows.
The content of a Young diagram is the number of its boxes, and we denote the content of the Young diagram $\nu$ by $\lvert\nu\rvert$.
If one considers the action of the center of $U(n_j)$, one sees that all diagrams $\nu_1,\dots,\nu_s$ such that $V^{U(n_1)}_{\nu_1}\otimes \cdots \otimes V^{U(n_s)}_{\nu_s}$ occurs in $\Sym^c(\C^N)$ have the same content $c$.
Thus the decomposition~\eqref{eq:kron sym} indeed provides a refinement of the grading of the symmetric algebra by degree.
Moreover, the diagrams $\nu_1,\dots,\nu_s$ all index irreducible representations $\pi_{\nu_j}$ of the symmetric group $\mathfrak S_c$.
By Schur-Weil duality, $g(\nu_1,\dots,\nu_s)$ is thus also equal to the multiplicity of the trivial representation of the symmetric group $\mathfrak S_c$ in the tensor product $\pi_{\nu_1}\otimes \cdots \otimes \pi_{\nu_s}$.

The numbers $g(\nu_1,\nu_2,\ldots,\nu_s)$ are called the \emph{Kronecker coefficients}.
The functions $k \mapsto g(k\nu_1,k\nu_2,\ldots,k\nu_s)$, $k\in \{0,1,2,\ldots\}$, are known as the \emph{dilated Kronecker coefficients} or \emph{stretching function} indexed by the Young diagrams $\nu_1,\dots,\nu_s$.
It follows from the $[Q,R]=0$ theorem, obtained by Meinrenken-Sjamaar~\cite{M-S}, that each dilated Kronecker coefficient is given by a \emph{quasi-polynomial function} for $k\geq1$.
In fact, there exist a decomposition of the space of $s$-tuples $(\nu_1,\dots,\nu_s)$ of Young diagrams with $n_j$ rows into polyhedral convex cones, each on which the function $g(\nu_1,\dots,\nu_s)$ is a quasi-polynomial.
We recall the general definition of a quasi-polynomial function and the precise quasi-polynomiality results in Section~\ref{subsec:quasipol}.
The Kronecker coefficients have a long history in representation theory and algebraic combinatorics.
More recently, new and important applications have arisen in quantum information and geometric complexity theory and much recent work on Kronecker coefficients has been motivated by this connection.
In both fields, computational methods are required for experimentation and checks of mathematical hypotheses.

\medskip

In this paper we shall present an algorithm that computes the function $g(\nu_1,\nu_2,\ldots ,\nu_s)$ locally.
More precisely, given as input a fixed $s$-tuple $(\nu_1^0,\nu_2^0,\ldots ,\nu_s^0)$, our algorithm produces a \emph{symbolic} function which coincides with $g(\nu_1,\nu_2,\ldots, \nu_s)$ in a conic neighborhood of this fixed $s$-tuple.
In particular, this allows us to symbolically compute the dilated Kronecker coefficients $k\mapsto g(k\nu_1,k\nu_2,\ldots ,k\nu_s)$ as a quasi-polynomial function of $k$.
A \textsc{Maple} implementation of our algorithm is available at~\cite{Kroneckerwebpage}.

We now describe briefly our approach.
By definition, the Kronecker coefficient $g(\nu_1,\dots,\nu_s)$ amounts to a branching problem, namely it is equal to the multiplicity of an irreducible representation of $U(n_1)\times\dots\times U(n_s)$ in $\Sym^c(\C^N)$, which is an irreducible representation of $U(N)$.
For $s=2$, the decomposition of $\Sym(\C^{n_1}\otimes\C^{n_2})$ is known and given by the Cauchy formula (see Lemma~\ref{lem:Cauchy formula}).
For $s>2$, we can use the Cauchy formula to reduce the number of factors $s$ by one.
Without loss of generality, we may assume that $n_1 = \max n_j$ and $n_1\leq M=n_2\cdots n_s$ (see Lemma~\ref{lem:Cauchy reduction}).
Using the Cauchy formula, we can write
$$\Sym(\C^N)=\Sym(\C^{n_1}\otimes \C^M)=\bigoplus_{\nu_1} V_{\nu_1}^{U(n_1)} \otimes V_{\tilde\nu_1}^{U(M)}$$
where $\nu_1$ indexes the finite dimensional irreducible polynomial representations of $U(n_1)$ and $\tilde{\nu_1}$ denotes the polynomial representation of $U(M)$ obtained from $\nu_1$ by adding more zeros on the right.
Now consider the homomorphism from $K=U(n_2)\times\cdots\times U(n_s)$ to $G=U(M)$ given by the tensor product.
If we restrict an irreducible representation $V^{U(M)}_{\tilde\nu_1}$ of $G$ to $K$, we obtain a decomposition
$$V_{\tilde\nu_1}^{U(M)}=\bigoplus_{\nu_2,\dots,\nu_s} m_{G,K}(\tilde{\nu_1}, \nu_2\otimes\cdots \otimes\nu_s) \, V_{\nu_2}^{U(n_2)}\otimes \cdots \otimes V_{\nu_s}^{U(n_s)},$$
where $m_{G,K}(\tilde{\nu_1}, \nu_2\otimes\cdots \otimes\nu_s)$ is the \emph{branching coefficient} computing the multiplicity of the representation $V_{\nu_2}^{U(n_2)}\otimes\cdots\otimes V_{\nu_s}^{U(n_s)}$ in the restriction. 
By comparing with \eqref{eq:kron sym}, we obtain
\begin{equation}\label{eq:kron via branching}
g(\nu_1,\ldots, \nu_s)=\begin{cases}
m_{G,K}(\tilde{\nu_1}, \nu_2\otimes\cdots \otimes\nu_s) & \text{if $\lvert\nu_1\rvert=\cdots=\lvert\nu_s\rvert$}, \\
0 & \text{otherwise.}
\end{cases}
\end{equation}
Thus the objective of the paper is to give an explicit formula for the branching coefficient and to propose an algorithm that implements it.

We state our explicit formula for $m_{G,K}(\lambda,\mu)$ in Theorem~\ref{theo:branchingsingular} in the more general context of branching rules.
Our starting point is the notion of a \emph{Jeffrey-Kirwan residue}, which allows us to produce a symbolic function of $(\lambda,\mu)$ that coincides with $m_{G,K}(\lambda,\mu)$ in a conic neignborhood of a given pair $(\lambda^0,\mu^0)$.
This problem is part of the more general problem of symbolically computing multiplicities.
In addition, we use the results of Meinrenken-Sjamaar~\cite{M-S} on the piecewise quasi-polynomial behavior of multiplicity functions in order to justify our method.
Our algorithm then is based on a variation of the Kostant formula for weight multiplicities~\cite{Kos}, together with the iterated residues formula for computing partition functions due to Szenes-Vergne~\cite{SzeVer}.

\medskip

Some of the results herein were presented for the first time at the conference on `Quantum Marginals' at the Isaac Newton Institute in Cambridge~\cite{Vergne-cambridge}.
The preprint~\cite{BV2015} is an extended version of the talk of the second author at this conference and is not intended to be published.
However, it may be interesting to the reader to consult that paper, since various aspects of the theory that come into play (Hamiltonian geometry, convexity, quasi-polynomial behavior, Jeffrey-Kirwan residues) are presented in~\cite{BV2015}, together with an extended bibliography.
In contrast, our focus in the present article is on the application of our general methods for computing branching coefficients to the particularly challenging case of Kronecker coefficients.

Motivated by discussions with M.~Walter, we consider also here the important special case of rectangular Kronecker coefficients, and adapted our algorithm for this case.
A new highly optimized Maple code,  \cite{Kroneckerwebpage},  was then written by M.~Walter and it is presented here in the Appendix.
\subsection{Examples}

Throughout the text we have included many examples.
Some of them are known and show the consistency of our results with other techniques of computation.
In addition, we give several new examples to show the power, and the limitations, of our computational approach.

For instance, in Section~\ref{exa:632} we consider the case of $n_1=6,n_2=3,n_3=2$.
Then, given as input a point $(\alpha^0,\beta^0,\gamma^0)$, where $\alpha^0$ is a Young diagram with  $6$ rows, $\beta^0$  with $3$ rows and $\gamma^0$ with $2$ rows, we effectively compute the Kronecker coefficient $g(\alpha,\beta,\gamma)$ symbolically in a conic neighborhood of the given $(\alpha^0,\beta^0,\gamma^0)$.
We also compute the dilated Kronecker coefficient $g(k\alpha,k\beta,k \gamma)$ with $\alpha,\beta,\gamma$ Young diagrams with $3$ rows.
For general $\alpha,\beta,\gamma$ with $3$ rows, we obtain a quasi-polynomial of degree $11$, with coefficients that are periodic functions of period at most $12$ in $k$ (see Section~\ref{C333}).
When $\alpha,\beta,\gamma$ are special, the degree might be much smaller.
We refer to Section~\ref{subsec:dilated} for several other examples of dilated Kronecker coefficients.

Rectangular Kronecker coefficients, i.e.\ those involving partitions of rectangular shape, are of special interest because of their direct relation with invariant theory:
If we consider the group $U(n_1)\times U(n_2)\times\cdots\times U(n_s)$ acting on $\C^{n_1}\otimes\C^{n_2}\otimes\cdots \otimes\C^{n_s}$ as well as characters $\chi_j=[p_j,\ldots,p_j]$ for $j=1,\dots,s$ with equal content $c=p_jn_j$, then
$g(k\chi_1,k\chi_2,\ldots, k\chi_s)$
is equal to
\[ \dim \left[ \Sym^{kc}\bigl( \C^{n_1}\otimes\C^{n_2}\otimes\cdots\otimes\C^{n_s}\bigr) \right]^{\SL(n_1)\times SL(n_2)\times\cdots\times SL(n_s)}. \]
Thus the power series
$R(t)=\sum_{k=0}^{\infty} g(k\chi_1,k\chi_2,\ldots, k\chi_s) t^{k}$
is the \emph{Hilbert series} of the ring of invariant polynomials under the action considered (see Example~\ref{Hilbertseries}).
Particularly challenging examples are the Hilbert series for $5$ qubits given in~\cite{LUTHI} and the Hilbert series of measures of entanglement for $4$ qubits computed in~\cite{Nolan}.
We compute these (and correct a misprint in~\cite{LUTHI}) using our techniques in Section~\ref{subsec:HS}, see in particular Section~\ref{Wallach.ex}.
We recently learned of a different, yet-unpublished, approach of M.~Grassl and R.~Zeier for computing Hilbert series, which is more efficient than our method for the $5$ qubits example (cf.~\cite{Grassl,Zeier}).
We conclude with two explicit examples, first a classical one~\cite{Kac,Vinberg,Kac-Vinberg}, then a new one.

\begin{example}
The dilated Kronecker coefficient
\[ m(k) = g(k[1,1,1],k[1,1,1],k[1,1,1]) \]
corresponds to the Hilbert series of the ring of invariants for the action of $SL(\C^3)\times SL(\C^3)\times SL(\C^3)$ on $\C^3\otimes \C^3\otimes \C^3$.
An efficient way to represent periodic functions is by using step polynomials (see~\cite{BBDKV,Verdo}).
In this representation, $m(k)$ is given by the following quasi-polynomial:
\begin{align*}
m(k) &= 1-\frac{3}{2}\fractional{\frac{1}{3}k}+\frac{3}{2}\fractional{\frac{1}{3}k}^2-\frac{3}{2}\fractional{\frac{1}{2}k}-\fractional{\frac{3}{4}k}^2 \\
&+\fractional{\frac{3}{4}k}\fractional{\frac{1}{2}k}+\fractional{\frac{1}{2}k}^2 +\left(\frac{1}{4}-\frac{1}{4}\fractional{\frac{1}{2}k}\right)k+\frac{1}{48}k^2.
\end{align*}
Here we write $\fractional{s}=s-\floor{s} \in [0,1)$ for $s\in\R$, where $\floor{s}$ denotes the largest integer smaller or equal to $s$.
It is easy to check that the corresponding Hilbert series is given by
$$\sum_{k=0}^{\infty} m(k)t^k=\frac{1}{(1-t^2)(1-t^3)(1-t^4)}.$$
Thus the result of our algorithm agrees with the determination of the ring of invariants, $[\bigoplus_{k=0}^{\infty}\Sym^{3k}(\C^3\otimes \C^3\otimes \C^3)]^{SL(\C^3)\times SL(\C^3)\times SL(\C^3)}$, which is known to be freely generated with generators in degree $2,3,4$.  
\end{example}

\begin{example}
Another example is
\begin{align*}
  m(k) &= g(k[3,3,3,3],k[4,4,4],k[4,4,4]) \\
  &=\dim \left[\Sym^{12 k}(\C^4\otimes \C^3\otimes \C^3) \right]^{SL(\C^4)\times SL(\C^3)\times SL(\C^3)}.
\end{align*}
Our algorithm gives the Hilbert series
$$\sum_{k=0}^{\infty} m(k)t^k={\frac {1+{t}^{9}}{  \left( 1-t\right) \left( 1-{t}^{2} \right) ^{2}  \left(1- {t}^{3} \right)  \left(1- {t}^{4} \right)}}.$$
\end{example}

\subsection{Outline of the article}

In Section~\ref{sec:setup}, we consider the action of a connected compact Lie group $K$ on an finite-dimensional Hermitian space $\CH$ and the corresponding decomposition
$$\Sym(\CH)=\sum_\lambda m_K^{\CH}(\lambda) \, V_\lambda^K.$$
We also introduce the Kirwan cone, which is the asymptotic support of the multiplicity function $\lambda\mapsto m_K^{\CH}(\lambda)$.
The Kirwan cone is a polyhedral cone.
Each facet (that is, a face of codimension one) of this cone generates a hyperplane that we call a wall of $C_K(\CH)$.

In Section~\ref{sec:multiplicities}, we define the notions of topes and Orlik-Solomon bases.
We discuss quasi-polynomial functions and recall the Szenes-Vergne iterated residue formula for the piecewise quasi-polynomial function $m_K^{\CH}(\lambda)$ in the case that $K$ is a torus.
We then discuss the Meinrenken-Sjamaar theorem on the piecewise quasi-polynomial behavior of $m_K^{\CH}(\lambda)$, when $K$ is a compact connected Lie group.
This theorem is important for our present work, since it allows us to compute $m_K^\CH(\lambda)$ by a deformation argument.
We also state some corollaries on the degrees of the functions $k\mapsto m_K^{\CH}(k\lambda)$ on the interior of the Kirwan cone, as well as on its faces.

In Section~\ref{sec:branch}, we then consider the general branching problem:
Given a homomorphism $K\to G$ and an irreducible representation $V_\lambda^G$ of $G$, decompose it into irreducible representations of $K$.
That is, $V_\lambda^G=\bigoplus m_{G,K}(\lambda,\mu) V_\mu^K$.
In Theorem~\ref{theo:mGK}, we recall some of the qualitative properties of the function $m_{G,K}(\lambda,\mu)$ implied by the Meinrenken-Sjamaar theorem.
This allows us to likewise compute $m_{G,K}(\lambda,\mu)$ by deformation arguments.
We discuss the regions where $m_{G,K}(\lambda,\mu)$ is given by a quasi-polynomial function.
Our main result is the general branching formula in Theorem~\ref{theo:branchingsingular}.

In Section~\ref{sec:examples}, we list several explicit examples of dilated Kronecker coefficients and of Hilbert series computed using our method.

In Appendix~\ref{app:thealgoKro}, we summarize the algorithm for computing Kronecker coefficients in detail.
As discussed before, we in essence use the Kostant multiplicity formula, together with the method of iterated residues, to compute $m_{G,K}(\lambda,\mu)$ using Theorem~\ref{theo:branchingsingular}.
When the stabilizer of $\lambda$ in the Weyl group of $G$ is large, we explain how to take advantage of this situation.
Similarly as in~\cite{C2,C-D-W}, our method for computing $g(k\lambda_1,k\lambda_2,\dots,k\lambda_s)$ is effective for small $n_1,n_2,\dots,n_s$. 
More precisely, if we fix the number $n_i$ of rows then the algorithm runs in time polynomial in the input size (and, in practice, relatively quick).

\section{Towards a multiplicity formula: the setup}\label{sec:setup}

We start by introducing some general notation.
Let  $K$ be a compact connected Lie group and $T_K$ a maximal torus of $K$.
We denote by $\k $ and $\t_\k$ the corresponding Lie algebras.
We use the superscript $^*$ for dual spaces and the subscript $_\C$ for complexifications; for example, we write $\t_\k^*$ and $\k_\C$.
We denote by $\CW_\k$ the Weyl group and by $w\mapsto \epsilon(w)=\det_{\t_\k} w$ its determinant representation.
The weight lattice $\Lambda_K$ of $T_K$ is a lattice in $i\t_\k^*$, $i=\sqrt{-1}$.
Any $\lambda\in \Lambda_K$ determines a one-dimensional representation of $T_K$ by $t\mapsto e^{\langle \lambda,X\rangle}$, with $t=\exp X, \ X\in \t_\k$.
As $\lambda$ takes imaginary values on $\t_\k$, $e^{\langle \lambda,X\rangle}$ is of modulus $1$.
We denote by $\Gamma_K\subset i\t_\k$ the dual lattice of $\Gamma_K$.
It consists of all elements $\gamma\in\Gamma_K$ such that, for all $\lambda\in \Lambda_K$, $\la \lambda,\gamma\ra$ is an integer.
Let $\Delta_\k\subset i\t_\k^*$ be the root system of $\k$ with respect to $\mathfrak t_\k$.
If $\alpha\in \Delta_\k$, its coroot $H_\alpha$ is in $i\t_\k$ and satisfies $\la \alpha,H_\alpha\ra=2$.
Fix a positive system $\Delta_\k^+$ for $\Delta_\k$, and let
$$i\t^*_{\k,\geq 0}=\left\{\xi, \langle \xi, H_\alpha\rangle \geq 0, \alpha \in \Delta_\k^+\right\}$$
denote the corresponding positive Weyl chamber.
Let $\rho_\k=\frac12\sum_{\alpha\in \Delta_\k^+} \alpha$ denote half the sum of the positive roots.
We write $\Lambda_{K,\geq 0}$ for the semigroup of dominant weights, that is, the set $\Lambda_K\cap i\t^*_{\k,\geq 0}$. 
We can parameterize the set of classes of irreducible finite-dimensional representations of $K$ by $\Lambda_{K,\geq 0}$:
given $\lambda \in \Lambda_{K,\geq 0}$, we denote by $V_\lambda^K$ the corresponding irreducible representation of $K$ with highest weight $\lambda$.
We denote by $\lambda^*$ the highest weight of the contragredient representation, $V_{\lambda^*}^K = (V_\lambda^K)^*$.
A dominant weight $\lambda$ is \emph{regular} if $\la\lambda,H_\alpha\ra\neq 0$ for all roots $\alpha\in \Delta_\g$.
Otherwise, we say that $\lambda$ is \emph{singular}.

When the group $K$ is understood, we abbreviate $\Lambda_K$ by $\Lambda$, $T_K$ by $T$, $\t_\k$ by $\t$, etc.
Finally, when dealing with several groups $K,G,\dots$, as in Section~\ref{sec:branch}, then we will use Fraktur letters $\k,\g,\dots$ for the Lie algebras and subscripts to distinguish the corresponding objects, e.g., $\Lambda_K,\Lambda_G,\t_\k,\t_\g,\ldots$.

\medskip

Since the unitary groups are important for the definition of the Kronecker coefficients in Eq.~\eqref{eq:kron sym}, we briefly pause to explain the notation just introduced for $K=U(n)$, the group of unitary $n\times n$ matrices.
We choose $T_{U(n)}\subset K$ as the set of diagonal unitary matrices.
Then the Lie algebra $\u(n)$ consists of the $n\times n$ anti-Hermitian matrices and $i\u(n)$ is the space of Hermitian matrices.
If we identify $\u(n)$ and $\u(n)^*$ via the bilinear form ${\rm Tr}(AB)$ then $\t_{\u(n)}=\t_{\u(n)}^*$  is the set of diagonal anti-Hermitian matrices.
The positive Weyl chamber is $i\t^*_{\u(n),\geq 0}=\{\xi=[\xi_1,\xi_2,\ldots, \xi_n]\}$ with $\xi_j\in \R$ and $\xi_1\geq \xi_2\geq \cdots \geq \xi_n$, where $\xi$ represents the Hermitian matrix with diagonal entries $\xi_1,\dots,\xi_n$.
The weight lattice is $\Lambda_{U(n)}=\{\lambda=[\lambda_1,\lambda_2,\ldots, \lambda_n]\}$ with $\lambda_j\in \Z$ and we write $\Gamma_{U(n)}\subset i\t$ for the dual lattice.
The semigroup of dominant weights is $\Lambda_{U(n),\geq 0}=\{\lambda=[\lambda_1,\lambda_2,\ldots, \lambda_n]\}$ with $\lambda_j\in \Z$ and $\lambda_1\geq \lambda_2\geq \cdots \geq \lambda_n$.

If  $\lambda \in\Lambda_{U(n),\geq 0}$ is such that  $\lambda_n\geq 0$, then $\lambda$ indexes a finite-dimensional irreducible \emph{polynomial} representation of $GL(n,\C)$.
The corresponding subset of $\Lambda_{U(n),\geq 0}$ will be denoted by $P\Lambda_{U(n),\geq 0}$ ($P$ for polynomial).
In this case, we may also identify $\lambda$ with a \emph{Young diagram} with no more than $n$ rows ($\lambda_j$ denotes the length of the $j$-th row).
As in the introduction, the \emph{content} $\lvert\lambda\rvert$ of the corresponding diagram is the number of boxes, that is, $\lvert\lambda\rvert=\sum_j\lambda_j$.
The dominant weight $[k,k,\ldots,k]$, corresponding to a \emph{rectangular} Young diagram with $n$ rows and $k$ columns, indexes the one-dimensional representation $g\mapsto\det(g)^k$ of $U(n)$.

Given $N\geq n$, there is a natural injection from $P\Lambda_{U(n),\geq 0}$ to $P\Lambda_{U(N),\geq 0}$ obtained by extending $\lambda$ with zeros on the right, and we denote by $\tilde{\lambda} = (\lambda,0,\dots,0)$ the highest weight so obtained.

Lastly, we similarly define the notation $Pi\t^*_{\u(n)}=\{\xi \in i\t^*_{\geq 0}, \xi_n\geq 0\}$, $\lvert\xi\rvert=\sum_j\xi_j$, and $\tilde{\xi}=(\xi,0,\dots,0)$ for $\xi\in Pi\t^*_{\u(n)}$.

\medskip

Let $\CH$ be a finite dimensional Hermitian vector space provided with a representation of $K$ by unitary transformations.
Assume (temporarily) that $K$ contains the subgroup of homotheties $\{e^{i\theta}\Id_\CH\}$.
Consider $\Sym(\CH)$, the space of symmetric tensors, so we have
\begin{equation}\label{eq:multiplicities}
  \Sym(\CH)=\bigoplus_{\lambda\in \Lambda_{K,\geq 0}} m_K^{\CH}(\lambda) \, V_\lambda^K,
\end{equation}
where $m_K^{\CH}(\lambda),$  the multiplicity of $V_\lambda^K$ in  $\Sym(\CH),$ is finite.
We also write  $$\Sym(\CH)=\bigoplus_{\mu} m_{T_K}^{\CH}(\mu) \, e^{\mu},$$  where $m_{T_K}^{\CH}(\mu)$ is the multiplicity of the weight $\mu.$
The multiplicities for $K$ and $T_K$ will be related in Lemma~\ref{lem:KandTmult}.

Before going on, we give two examples.
The first one is the main example we will be interested in this paper and the second is related to the computation of Hilbert series.

\begin{example}[Kronecker coefficients]\label{ex:kron}
  Consider $\CH=\C^{n_1}\otimes \cdots\otimes \C^{n_s}$ with the exterior tensor product action of $K=U(n_1)\times \cdots \times U(n_s)$.
  Given Young diagrams $\nu_j\in P\Lambda_{U(n_j),\geq0}$, $j=1,\dots,s$, the $s$-tuple $\lambda=(\nu_1,\dots,\nu_s)$ is a highest weight for $K$.
  The Kronecker coefficients~\eqref{eq:kron sym} are given by
  $g(\nu_1,\nu_2,\ldots, \nu_s) = m_K^\CH(\lambda).$
\end{example}

The case $s=2$ is well-known:

\begin{lemma}[Cauchy formula, {\cite[p.~64]{KP}}]\label{lem:Cauchy formula}
  Let $N, n$ be positive integers, and assume that $N\geq n$.
  Consider $\CH=\C^n\otimes \C^N$ with the action of $K=U(n)\times U(N)$.
  Then the decomposition of $\Sym(\C^{n}\otimes \C^N)$ with respect to $\U(n)\times U(N)$ is given by the \emph{Cauchy formula}
  \begin{equation}\label{exa:Cauchy formula}
    \Sym(\C^{n}\otimes \C^N)=\bigoplus_{\nu_1 \in P\Lambda_{U(n),\geq 0}}V_{\nu_1}^{U(n)}\otimes V_{\tilde\nu_1}^{U(N)}.
  \end{equation}
  In other words, $g(\nu_1,\nu_2)=\delta_{\tilde\nu_1,\nu_2}$.
\end{lemma}

As mentioned in the introduction, the calculation of Kronecker coefficients can be reduced in the following way:

\begin{lemma}\label{lem:Cauchy reduction}
  Let $n_1,\dots,n_s$ be positive integers, and $M=n_2 n_3\cdots n_s$.
  Let $\nu_j \in P\Lambda_{U(n_j),\geq0}$ for $j=1,\dots,s$.
  \begin{enumerate}
  \item[(i)] The Kronecker coefficients are invariant under permuting the Young diagrams: $g(\nu_1,\dots,\nu_s)=g(\nu_{\pi(1)},\dots,\nu_{\pi(s)})$ for $\pi\in\mathfrak S_s$.
  \item[(ii)] If $n_1>M$ then the Kronecker coefficients stabilize in the sense that $g(\nu_1,\dots,\nu_s)$ is non-zero only if $\nu_1=\tilde\eta_1$, where $\tilde\eta_1$ is obtained from an element $\eta_1\in P\Lambda_{U(M),\geq0}$ by adding zeros on the right. In this case,
    $g(\nu_1,\nu_2,\ldots,  \nu_s)=g(\eta_1, \nu_2,\ldots, \nu_s).$
  \end{enumerate}
  As a consequence we may assume that $n_1=\max_j n_j$ and $n_1\leq M$ when computing Kronecker coefficients.
\end{lemma}

\begin{example}[Hilbert series]\label{Hilbertseries}
  Let $\k=\z \oplus [\k,\k]$, and assume that the center $\z=\R J$ of $\k$ acts by homotheties on $\CH$.
  Consider an element $\chi\in \Lambda_K$ such that $\chi(iJ)=1$ and $\chi=0$ on $i(\t_\k\cap [\k,\k])$.
  Then,
  \[ m_K^{\CH}(k\chi)=\dim \left[ \Sym^k(\CH) \right]^{[K,K]}, \]
  and it follows that, by definition, the series $R(t) = \sum_{k=0}^\infty m_K^\CH(k\chi) \, t^k$ is the \emph{Hilbert series} of the ring of invariant polynomials under the action of $[K_\C,K_\C]$.
\end{example}

\begin{remark}\label{rem:gorenstein}
  The ring of invariants is a Gorenstein ring, so its Hilbert series is of the form
  $$R(t)=\frac{P(t)}{\prod_{j=1}^N (1-t^{a_j})},$$
  where $P(t)$ is a palindromic polynomial~\cite{Mu}.
  Furthermore, it follows from~\cite{M-S} that the degree of $P(t)$ is strictly less than $\sum_j a_j$.
\end{remark}

The action of $K$ on $\CH$ admits a \emph{moment map} (in the sense of symplectic geometry~\cite{Guillemin})
$\Phi_K:\CH\to i\k^*$ given by
$$\Phi_K(v)(X)=\langle Xv,v\rangle,  v\in \CH, X\in \k,$$
where we note that $\langle Xv,v\rangle$ is purely imaginary since $K$ acts unitarily on $\CH$.
We consider $i\t^*_{\k,\geq 0}$ as a subset of $i\k^*$.
Then the \emph{Kirwan cone} is defined as the intersection of the image of the moment map with the positive Weyl chamber:
\begin{equation}\label{eq:kirwan cone}
  C_K(\CH)=\Phi_K(\CH)\cap i\t^*_{\k,\geq 0}
\end{equation}
Kirwan's convexity theorem implies that $C_K(\CH)$ is a rational polyhedral cone.
The cone $C_K(\CH)$  is related to the multiplicities through the following basic result, which is a particular case of Mumford's theorem~\cite{Mum} (a proof following closely Mumford's argument can be found in~\cite{Be}):

\begin{proposition}[Mumford]\label{prp:mumford}
  If $\lambda\notin C_K(\CH)$ then $m_K^{\CH}(\lambda)=0$.
  Conversely, if $\lambda$ is a dominant weight belonging to $C_K(\CH)$, there exists an integer $k>0$ such that $m_K^{\CH}(k\lambda)$  is non-zero.
\end{proposition}

Thus the support of the function $m_K^{\CH}(\lambda)$ is contained in the Kirwan cone $C_K(\CH)$ and its asymptotic support is exactly $C_K(\CH)$.

\begin{remark}
  As $C_K(\CH)$ is a rational polyhedral cone, it can in principle be described by a finite number of inequalities determined by elements $X_a\in \Gamma_K$ as:
  $C_K(\CH)=\{\xi\in i\t^*_{\k,\geq 0}, \langle X_a,\xi\rangle \geq 0 \, \forall a\}.$
  It is in general quite difficult to determine these inequalities explicitly.
  An algorithm to describe the inequalities of $C_K(\CH)$, based on Ressayre's notion of dominant pairs~\cite{Ressayre-inventiones}, is given by Vergne-Walter~\cite{V-W}.
\end{remark}

Define $\CH_{pure}=\{v\in \CH, \langle v, v\rangle=1\}$, the set of elements of $\CH$ of norm~$1$.
The \emph{Kirwan polytope} is the rational polytope defined by
\[ \Delta_K(\CH)=\Phi_K(\CH_{pure}) \cap i\t_{\k,\geq0}^*. \]
Note that the Kirwan cone is the cone over the Kirwan polytope, $C_K(\CH)=\R_{\geq 0} \Delta_K(\CH)$.

\begin{example}[Qubits]\label{ex:N qubit cone}
In quantum mechanics, the elements of $\CH_{pure}$ are called pure states.
The defining representation $\C^2$ of $\SU(2)$ is known as a \emph{qubit}, and the exterior tensor power representation of $K=\SU(2)^{\times s}$ on $\CH=(\C^2)^{\otimes s}$ as \emph{$s$ qubits}.

Higuchi-Sudbery-Szulc~\cite{H-Su-Sz} have determined the Kirwan cone for $s$ qubits:
Consider $\xi_1=[\xi_1^1,\xi_2^2],\ldots, \xi_s=[\xi_1^s,\xi_2^s]$, a sequence of $s$ elements of $Pi\t^*_{\u(2),\geq 0}$ (that is $\xi_j^1\geq \xi_j^2\geq 0$).
Then $(\xi_1,\ldots,\xi_s)\in C_K(\CH)$ if and only if, for any $j=1,2,\ldots,s,$ $\xi_j^2\leq \sum_{k\neq j} \xi_k^2$, as well as $\lvert\xi_1\rvert=\lvert\xi_2\rvert=\dots=\lvert\xi_s\rvert$, where we recall that $\lvert\xi_j\rvert=\xi_j^1+\xi_j^2$.
\end{example}

We conclude this section with one more example of the connection between multiplicities and Kirwan cone:

\begin{example}[Cauchy formula, revisited]\label{ex:Cauchy cone}
  As in Lemma~\ref{lem:Cauchy formula}, let $\CH=\C^n\otimes \C^N$ under the action of $K=U(n)\times U(N)$, where $N, n$ are positive integers such that $N\geq n$.
  Using the Hermitian inner product, we may identify $A\in\CH$ with a matrix $A:\C^n\to\C^N$.
  Then the moment map $\Phi_K\colon\CH\to i\k^*$ is given by $\Phi_K(A)=[AA^*,A^*A]$, with values in the spaces of Hermitian matrices of size $n$ and $N$, respectively.
  It is easy to see that the Kirwan cone is precisely the ``diagonal''
  \[ C_K(\CH) = \big\{ (\xi,\tilde\xi), \xi\in Pi\t^*_{\u(n),\geq0}. \big\} \]
  Thus the multiplicity function determined by the Cauchy formula~\eqref{exa:Cauchy formula} is supported exactly on the set $\Lambda_K\cap C_K(\CH)$ (and with value 1), in agreement with Proposition~\ref{prp:mumford}.
\end{example}

The Kirwan cone $C_K(\CH)$ is always contained in the real vector space spanned by the weights of the representation.
In the case of the Kronecker coefficients, Example~\ref{ex:kron}, this corresponds precisely to the equations $\lvert\xi_1\rvert=\dots=\lvert\xi_s\rvert$.
However, it is not always true that the cone $C_K(\CH)$ has non-empty interior in this vector space, as the preceding Example~\ref{ex:Cauchy cone} demonstrates (a special case is Example~\ref{ex:N qubit cone} for $s=2$).

\section{Multiplicities and partitions functions}\label{sec:multiplicities}

In this section, we first revisit the definition of iterated residues, topes, and Orlik-Solomon basis.
We then recall the quasi-polynomial nature of the multiplicities and close by discussing its behavior on the facets of the Kirwan cone.

\subsection{Iterated residues, topes, and Orlik-Solomon bases}

If $f$ is a meromorphic function in one variable $z$, consider its Laurent series $\sum_n a_n z^n$ at $z=0$.
The coefficient of $z^{-1}$ is denoted by $\Res_{z=0}f$.

We now recall the notion of an iterated residue.
Let $E$ be a real vector space, $r=\dim E$, equipped with a volume form $\det$.
Consider an ordered basis of $E$:
$\overrightarrow{\sigma}=[\alpha_1,\alpha_2,\ldots, \alpha_r].$
For $z\in E^*_\C,$ let $z_j=\langle z,\alpha_j\rangle$.
Then $(z_1,z_2,\ldots, z_r)$ are coordinates for $E^*_\C\sim \C^r$.
Given a meromorphic function $f$ on $E^*_\C$ with poles on a finite union of hyperplanes, use the coordinates obtained from $\overrightarrow\sigma$ to express it as a function $f(z)=f(z_1,\dots,z_r)$.
In particular, $f$ may have poles on $z_j=0$.
Then the iterated residue of $f$ associated to $\overrightarrow\sigma$ is defined as follows:

\begin{definition}[Iterated residue]\label{def:ires}
  Then the \emph{iterated residue} of $f$ with respect to the ordered basis $\overrightarrow\sigma$ is defined by
  \[
    \Res_{{\overrightarrow{\sigma}}}(f) = \frac 1 {\lvert\det\sigma\rvert} \Res_{z_1=0}(\Res_{z_2=0}\cdots(\Res_{z_r=0}f(z_1,z_2,\ldots,z_r))\cdots),
  \]
  where we recall that $\det$ refers to the volume form of $E$.
\end{definition}

\begin{remark}
  The iterated residue depends on the order of the basis elements in $\overrightarrow\sigma$.
\end{remark}

\begin{definition}[Admissible hyperplane, regular element, tope, basis]\label{def:admissible etc}
  Let $\Psi=[\psi_1,\dots,\psi_N]$ be a finite list of vectors in $E$.
  We say that a hyperplane $H\subset E$ is \emph{$\Psi$-admissible} if $H$ is generated by elements of $\Psi$.
  We denote by $\CA(\Psi)$ the set of admissible hyperplanes.

  We say that $\xi\in E$ is \emph{$\Psi$-regular} if $\xi$ doesn't belong to any of the admissible hyperplanes in $\CA(\Psi)$.
  We call $\a\subset E$ a \emph{$\Psi$-tope} if it is a connected component of the complement of the union of the $\Psi$-admissible hyperplanes.
  We will often omit the prefix $\Psi$ when the context is clear.
  Any regular element $\xi$ determines a tope that we denote by $\a(\xi)$.

  Lastly, if $\overrightarrow\sigma$ is a sublist of $\Psi$ such that its elements form a basis of $E$ then we say that $\sigma$ is an \emph{ordered basis} of $\Psi$.
  The underlying set $\sigma$ will be called simply a \emph{basis} of $\Psi$.
  We write $\Cone(\sigma)$ for the convex cone generated by the elements of $\sigma$.
\end{definition}

\begin{example}\label{exE}
  Let $E=\R\epsilon_1 \oplus \R\epsilon_2$ denote a two-dimensional vector space with basis vectors $\epsilon_1,\epsilon_2$.
  Consider the list of vector $\Psi=[\psi_1,\psi_2,\psi_3]=[\epsilon_1,\epsilon_2,\epsilon_1+\epsilon_2]$.
  Then $\{ [\psi_1,\psi_2], [\psi_1,\psi_3], [\psi_2,\psi_3] \}$ is the set of ordered bases of $\Psi$.
\end{example}

From now on we will assume that the elements of $\Psi$ generate a lattice $L \subseteq E$.

\begin{definition}[Index]\label{def:index}
  For $\sigma$ a basis of $\Psi$, let $d_\sigma$ be the smallest integer such that $d_\sigma L$ is contained in the lattice $\Z \sigma$ generated by $\sigma$.
  Then we define the \emph{index} of $\Psi$ with respect to $L$ as
  \[ q(\Psi) = \operatorname{lcm} \{ d_\sigma, \sigma \text{ basis of } \Psi \}, \]
  where $\operatorname{lcm}$ denotes the least common multiple.
\end{definition}

We now define the central concept of an Orlik-Solomon basis.

\begin{definition}[Orlik-Solomon basis]\label{OS}
  Let $\overrightarrow\sigma=[\psi_{i_1},\psi_{i_2},\dots,\psi_{i_r}]$ be an ordered basis of $\Psi$.
  Then $\overrightarrow\sigma$ is an \emph{Orlik-Solomon (OS) basis} if, for each $l=1,\dots,r$, there is no $1\leq j<i_l$ such that the elements $\{\psi_j,\psi_{i_l},\dots,\psi_{i_r}\}$ are linearly dependent.
  Note that the notion of an OS basis depends on the order of $\Psi$.

  We write $\mathcal{OS}(\Psi)$ for the set of OS bases.
  In addition, if $\a$ is a $\Psi$-tope then we denote
  \[ \mathcal{OS}(\Psi,\a) = \{ \overrightarrow\sigma \in \mathcal{OS}(\Psi), \a \subset \Cone(\sigma) \}. \]
  Its elements are called the \emph{OS bases adapted to $\a$}.

\end{definition}

\begin{example}
  We continue with Example~\ref{exE}.
  We compute $\mathcal{OS}(\Psi)=\{[\psi_1,\psi_2], [\psi_1,\psi_3]\}$.
  Moreover, $\a=\R_{>0}\epsilon_2\oplus \R_{>0}(\epsilon_1+\epsilon_2)$ is a $\Psi$-tope, and $\mathcal{OS}(\Psi,\a)=\{ [\psi_1,\psi_2]\}$.
\end{example}

For an algorithm for computing $\mathcal{OS}(\Psi,\a(\xi))$, given as input a $\Psi$-regular element $\xi$, we use the method described in~\cite[Section 4.9.6]{BV2015} (see Appendix~\ref{app:thealgoKro}).
It is based on the notion of \emph{maximal nested sets} of De Concini-Procesi~\cite{DCP} and developed in~\cite{BBCV}.

\subsection{Quasi-polynomial functions and multiplicities}\label{subsec:quasipol}

We now describe the nature of the function $m_K^{\CH}(\lambda)$ on $C_K(\CH)$, defined in Eq.~\eqref{eq:multiplicities}.
In particular, the results apply to the Kronecker coefficients, as explained in Example~\ref{ex:kron}.

Let $L$ be a lattice in a real vector space $E$.
Given an integer $q$, a function $c$ on $L$ has \emph{period} $q$ if $c(\lambda+q \nu)=c(\lambda)$ for all $\lambda,\nu$ in $L$.
We say that $c$ is a \emph{periodic function} on $L$ if there exists a $q$ such that $c$ has period $q$.
Note that, if $q=1$, then $c$ is a constant function on $L$.

Let $\lambda_0\in L$ and $q$ an integer.
The restriction of a polynomial function on $E$ to a coset $\lambda_0+qL$ will be called a \emph{polynomial function} on the coset $\lambda_0+qL$.

\begin{definition}[Quasi-polynomial function]
  A \emph{quasi-polynomial function} $p$ on $L$ is a linear combination of products of polynomial functions with periodic functions.
  In other words, a quasi-polynomial function $p$ can be written as
  \[ p(\lambda) = \sum_i c_i(\lambda) p_i(\lambda), \]
  where the $p_i$ are polynomial functions and the ``coefficients'' $c_i$ are periodic functions on $L$.

  We say that $p$ has \emph{period} $q$ if all the $c_i(\lambda)$ have period $q$.
  In this case, for any $\lambda_0\in L$, the function $\lambda\mapsto p(\lambda_0+q\lambda)$ is a polynomial function on $L$.
  It follows that we can represent a quasi-polynomial function of period $q$ as a family of polynomials, indexed by $L/qL$.
\end{definition}

If $q$ is very large then the above description is not very efficient, as the number of cosets can be quite large.
In the present work, we will only have to consider relatively small periods $q$.

The space of quasi-polynomial functions is graded:
we say that $p$ is \emph{homogeneous} of degree $k$ if 
the polynomials $p_j$ are homogeneous of degree $k.$ 
 As for polynomials, we say that $p$ is of \emph{degree} $k$ if $p$ is a sum of homogeneous terms of degree less or equal to $k$, and the term of degree $k$ is non-zero.

\begin{example}\label{ex:simple quasipol}
  The function
  $$m(k)=\frac12 k^2+k+\frac34+\frac14(-1)^k$$
  is a quasi-polynomial function of $k\in \Z$, of degree $2$ and period $2$.
  On each of the $2$ cosets, $m(k)$ coincides with a polynomial, namely
  \[
  m(k) = \begin{cases}
  m_0(k)=\frac12 k^2+k+\frac34+\frac14 & \text{if $k=0\pmod2$}, \\
  m_1(k)=\frac12 k^2+k+\frac34-\frac14 & \text{if $k=1\pmod2$}.
  \end{cases}
  \]
\end{example}

In practice, we will naturally obtain a quasi-polynomial $p$ as a sum of quasi-polynomial functions $p_1,p_2,\ldots,p_u$ of periods $q_1,q_2,\ldots, q_{u}$.
Thus $p$ is of period $q$, where $q$ is the least common multiple of $q_1,q_2,\ldots,q_u$.
Furthermore, in our examples, when the period $q_i$ is large, the degree of the corresponding quasi-polynomial $p_i$ is usually small.
Thus it is more efficient to keep $p$ as represented as $\sum p_i$, the number of cosets needed to describe each $p_i$ being $q_i$, since $\sum q_i$ is usually much smaller that $q$.
We will thus refer to the set $\{q_1,q_2,\ldots,q_{u}\}$ as a \emph{set of periods} of the quasi-polynomial function $p$.

\begin{example}
  The quasi-polynomial $m(k)$ in Example~\ref{ex:simple quasipol} is the sum of the polynomials
  \[ p_1(k)=\frac12 k^2+k+\frac34, \quad p_2(k) = \frac14(-1)^k. \]
  The first has degree 2 and period 1 (i.e., it is an ordinary polynomial), while the second has degree 0 and period 2.
  Thus $m$ has a set of periods $\{1,2\}$.
\end{example}

As we will discuss below, the dilated Kronecker coefficients are quasi-polynomial functions.
We do not in general know the set of periods as a function of the number of rows $n_1,\dots,n_s$ of the Young diagrams.
However, here are some concrete examples obtained by using our algorithm (see Section~\ref{sec:examples} for more detail):

\begin{example}\label{ex:periods}
For $n_1=n_2=n_3=3$, the function $k\mapsto g(k\lambda,k\mu,k\nu)$ is a quasi-polynomial function with set of periods included in $\{1,2,3,4\}$, leading to polynomial behavior on the cosets $f+12\Z$.

For the case of four qubits, $n_1=n_2=n_3=n_4=2$, the function $k\mapsto g(k\nu_1, k\nu_2, k\nu_3, k\nu_4)$ is a quasi-polynomial function with set of periods included in $\{1,2,3\}$, leading to polynomial behavior on the cosets $f+6\Z$.

For the case of five qubits, $n_1=n_2=n_3=n_4=n_5=2$, the function $k\mapsto g(k\nu_1, k\nu_2, k\nu_3, k\nu_4, k\nu_5)$ is a quasi-polynomial function with set of periods included in  $\{1,2,3,4,5\}$ leading to polynomial behavior on cosets $f+60\Z$.
\end{example}

In the case of one variable, we can give the following characterization of quasi-polynomiality.
If the function $p(k)$ is quasi-polynomial, its generating series $\sum_{k=0}^{\infty} p(k) t^k$ is the Taylor expansion at $t=0$ of a rational function
\[ R(t)=\frac{P(t)}{\prod_{i=1}^s(1-t^{a_i})}, \]
where the $a_i$ are integers, and $P(t)$ a polynomial in $t$ of degree strictly less than $\sum_i a_i$.
The correspondence is  as follows.
Consider   a quasi-polynomial $p(k)$ of period $q$, equal to $0$ on all cosets except the coset $f+q\Z$, with $0\leq f< q$.
Write  the polynomial function   $j\mapsto p(f+qj)$  of degree $R$ in terms of binomials: $p(f+qj)=\sum_{n=0}^R a(n) \binom{j+n}{n}$.
Then $$\sum_{j=0}^{\infty} p(f+qj) t^{f+qj}=t^f\sum_{n=0}^Ra(n)\frac{1}{(1-t^q)^{n+1}}.$$
In our examples, the degree of the quasi-polynomial function $p$, as well as its period, will not be very large, so there is no computational difficulty in obtaining the rational function $R(t)$ starting from $p(k)$, and conversely.
We give a striking example of the function $R$, giving the Hilbert series of measures of entanglement for $4$ qubits~\cite{Nolan} and the corresponding quasi-polynomial $p(k)$ in Section~\ref{Wallach.ex} (cf.\ Example~\ref{Hilbertseries}).

\medskip

We now return to the general setting of $K$ acting on an $N$-dimensional Hermitian vector space $\CH$ by unitary transformations.
Let $\bar K$ denote the image of $K$ in $U(\CH)$.
For the rest of this chapter, we choose $L=\Lambda_{\bar K}$ and $E = i\t^*_{\bar\k}$.
Since $\bar K$ is a quotient of $K$, we may think of $L \subseteq \Lambda_K$ and $E \subseteq i\t^*_{\k}$.
We equip $E$ with a volume form such that the fundamental cell of $L$ has unit measure.
We choose an order on the weights for the action of the maximal torus $T_K$ and write
$$\Psi=[\psi_1,\psi_2,\ldots,\psi_N]$$
with each $\psi_i\in \Lambda_{\bar K}\subseteq\Lambda_K\subset E$.
Note that $\Psi$ generates the lattice $L=\Lambda_{\bar K}$ (essentially by definition).
Now, we do no longer assume that the action of $K$ contains the homotheties $e^{i\theta} {\rm \Id}_\CH$.
Instead, we will require that the cone $\Cone(\Psi)$ generated by $\Psi$ is a pointed cone:
$\Cone(\Psi)\cap-\Cone(\Psi)=\{0\}$.
This condition ensures that the multiplicities for the action of $T_K$ are finite.

\begin{example}
  In the situation of Example~\ref{ex:kron}, where $K=U(n_1)\times\dots\times U(n_s)$ acts on $\CH=\C^{n_1}\otimes\dots\otimes\C^{n_s}$, we can identify $\bar K$ with the quotient $K/Z$, where $Z$ is the subgroup of the center that consist of the elements $(z_1\Id,\dots,z_s\Id)$ with $z_j\in U(1)$ and $z_1\cdots z_s=1$.
  The lattice $L=\Lambda_{\bar K}$ can be identified with the set of $\lambda=(\lambda_1,\dots,\lambda_s) \in \Lambda_K$ that satisfy $\lvert\lambda_1\rvert=\dots=\lvert\lambda_s\rvert$.
\end{example}

Let $\CP_\Psi$ be the function on $\Lambda_K$  that computes the number of ways we can write $\mu\in \Lambda_K$ as $\sum x_i\psi_i$ with $x_i$ nonnegative integers.
The function $\CP_\Psi(\mu)$ is known as the \emph{(Kostant) partition function} with respect to $\Psi$.
It is immediate to see that
\begin{equation}\label{eq:easy kostant}
  m_{T_{K}}^{\CH}(\mu)=\CP_\Psi(\mu).
\end{equation}
Moreover, the cone $C_{T_K}(\CH)$ is just the cone $\Cone(\Psi)$ generated by the list $\Psi$ of weights.
We now record the following fundamental relation between the multiplicities for $T_K$ and for $K$, which follows from the Weyl character formula:


\begin{lemma}\label{lem:KandTmult}
  Let $\lambda\in\Lambda_{K,\geq 0}$ be a dominant weight.
  Then:
  \[ m_K^{\CH}(\lambda)=\sum_{w\in \CW_\k} \epsilon(w) \CP_\Psi(\lambda+\rho_\k-w(\rho_\k))=\sum_{w\in \CW_\k} \epsilon(w) m_{T_K}^{\CH}(\lambda+\rho_\k-w(\rho_\k)), \]
  where we recall that $\CW_\k$ denotes the Weyl group and $\rho_\k$ half the sum of the positive roots.
\end{lemma}

\begin{remark}
  For each $y\in i\t_\k^*$, define the polytope
  $$\Pi_\Psi(y)=\Big\{[x_1,\ldots,x_N] \in \R^N, x_i\geq 0, \sum_{a=1}^N x_a\psi_a=y\Big\}.$$
  Then $m_{T_K}^{\CH}(\mu) = \mathcal P_\Psi(\mu)$ ($\mu\in \Lambda_K$) is the number of integral points in the polytope $\Pi_\Psi(\mu)$.
  Christandl-Doran-Walter~\cite{C-D-W} compute the multiplicities $m_K^{\CH}(\lambda)$ by using the formula in Lemma~\ref{lem:KandTmult} together with the preceding observation.
  Indeed, the multiplicities for $T_K$ are given by the number of integral points in the polytopes $\Pi_\Psi(\lambda+\rho_\k-w(\rho_\k))$, and they employ Barvinok's algorithm~\cite{Bar}, as implemented in~\cite{Verdo,latte}, to count these efficiently.
  Their method is of polynomial runtime in the input data when $K$ and $\CH$ are fixed (for the Kronecker coefficients, when the numbers of rows are fixed).
\end{remark}

We now describe our own approach to computing multiplicities based on the iterated residues formula from Szenes-Vergne~\cite{SzeVer}.
For $z\in (\t_{\k})_\C$, define:
$$S_{T_K}^{\Psi}(\mu,z)=e^{\langle\mu,z\rangle }\frac{1}{\prod_{\psi\in \Psi} (1-e^{-\langle\psi,z\rangle })}$$
The relationship between $S_{T_K}^\Psi$ and the multiplicities for $T_K$ is as follows:
If $\Re \langle\psi,z\rangle>0$ for all $\psi\in\Psi$ then it is clear from the geometric series and Eq.~\eqref{eq:easy kostant} that
\begin{equation}\label{eq:geometric}
  \frac{1}{\prod_{\psi \in \Psi}(1-e^{-\langle\psi,z\rangle})}
=\!\!\!\! \sum_{\nu \in \Cone(\Psi)}\!\!\CP_\Psi(\nu) e^{-\langle \nu,z\rangle}
=\!\!\!\! \sum_{\nu \in \Cone(\Psi)}\!\!m_{T_K}^\CH(\nu) e^{-\langle \nu,z\rangle}.
\end{equation}
Thus $S_{T_K}^\Psi(\mu,z)$ is the Laplace transform of translates of the multiplicity function $m_{T_K}^\CH$.
We now explain how the multiplicities can be recovered by using iterated residues.

\smallskip

Denote by $\Gamma_{\bar K}$ the dual lattice of $\Lambda_{\bar K}$ and by $q(\Psi)$ the index of $\Psi$ with respect to $L=\Lambda_{\bar K}$ (see Definition~\ref{def:index}).
So, if $\sigma$ is a basis of $\Psi$ and $q$ a multiple of $q(\Psi)$, then $q \Lambda_{\bar K} \subset \Z\sigma$.

\begin{definition}\label{def:szenes vergne quasipol}
  Let $\a$ be a $\Psi$-tope and $q$ a multiple of the index $q(\Psi)$.
  Define
  $$p_\a^\Psi(\mu)=\sum_{\gamma\in \Gamma_{\bar K}/q\Gamma_{\bar K}} \sum_{{\overrightarrow{\sigma}}\in \mathcal{OS}(\Psi,\a)} \Res_{\overrightarrow{\sigma}} S_{T_K}^{\Psi}(\mu, z+\frac{2i\pi}{q} \gamma).$$
\end{definition}

\begin{remark}\label{rem:szenes vergne quasipol}
  The role of $\a$ in the definition of $p_\a^\Psi(\mu)$ is to select the set $\mathcal{OS}(\Psi,\a)$, that is, the paths along which to calculate the iterated residue in $z$ of the function $z\mapsto S_{T_K}^{\Psi}(\mu, z+\frac{2i\pi}{q} \gamma)$.
  The function $p_\a^\Psi(\mu)$ is independent of the choice of the multiple $q$ of $q(\Psi)$.
  Thus, in principle, it is most economical to choose $q=q(\Psi)$.
  However, we will later have to compute several such partitition functions and it will be more convenient to choose a $q$ that works for all of them (i.e., one that is a multiple of all the $q(\Psi)$ involved).

  If $\gamma\in \Gamma_{\bar K}$, the iterated residue of the function $z\mapsto S_{T_K}^{\Psi}(\mu, z+\frac{2i\pi}{q} \gamma)$ is a quasi-polynomial function of $\mu$ and period $q$.
  Indeed, the residue depends on $\mu$ through the Taylor series at $z=0$ of $e^{\langle\mu,z+\frac{2i\pi}{q} \gamma\rangle}=e^{\frac{2i\pi}{q}\langle\mu,\gamma\rangle} e^{\langle\mu,z\rangle}$, and $e^{\frac{2i\pi}{q} \langle\mu,\gamma\rangle}$ is a periodic function of $\mu$ of period $q$.
\end{remark}

The key result that we will use is the following theorem, which asserts that the multiplicities for $T_K$ (equivalently, by Eq.~\eqref{eq:easy kostant}, the values of the Kostant partition function) agree in the closure of any given tope $\a$ with the corresponding quasi-polynomial $p_\a^\Psi$ from Definition~\ref{def:szenes vergne quasipol}.

\begin{theorem}[Szenes-Vergne, {\cite{SzeVer}}]\label{thm:multT}
  Let $\a$ be a $\Psi$-tope such that $\a\subseteq\Cone\Psi$.
  Then, for all $\mu\in\overline\a \cap \Lambda_K$,
  $$m_{T_K}^{\CH}(\mu)=\CP_\Psi(\mu)=p_\a^\Psi(\mu).$$
\end{theorem}

\medskip

We now discuss a consequence of the Meinrenken-Sjamaar theorem~\cite{M-S} (see also \cite{Par-Ve} for a different proof).
It will later be important to justify the correctness of our algorithm.

Recall that a \emph{cone decomposition} of a rational polyhedral cone $C$ is a set $\{\c_1,\c_2,\ldots, \c_m\}$ of (closed) rational polyhedral cones such that
\begin{enumerate}
\item[(i)] $C=\bigcup_{i=1}^m \c_i$,
\item[(ii)] $\c_1,\c_2,\ldots,\c_m$ all have the same dimension $\dim C$,
\item[(iii)] $\c_1,\c_2,\ldots,\c_m$ intersect along faces (that is $\c_a\cap \c_b$ is a face of both $\c_a$ and $\c_b$).
\end{enumerate}

\begin{theorem}[Meinrenken-Sjamaar, {\cite{M-S}}]\label{theo:multH}
  There exists a cone decomposition $C_K(\CH)=\bigcup_a \c_a$ and, for each $a$, a quasi-polynomial function $p_{K,a}^\CH$ on the lattice $\Lambda_K$, all of the same degree $d=\dim_\C \CH- \lvert\Delta_\k^+\rvert-\dim C_K(\CH)$, such that, for all $\lambda\in\overline{\c_a} \cap \Lambda_K$,
  $$m_K^{\CH}(\lambda)=p_{K,a}^{\CH}(\lambda).$$
\end{theorem}

\begin{remark}
  Theorem~\ref{theo:multH}, as well as Theorem~\ref{theo:mGK} below, are two particular instances of the $[Q,R]=0$ theorem of Mein\-ren\-ken-Sja\-maar~\cite{M-S}.
  This theorem gives a geometric formula for the quantization of a $K$-Hamiltonian manifold $M$.
  In Theorem~\ref{theo:multH}, $M=\CH$, while in Theorem~\ref{theo:mGK}, $M=T^*G$.
  For more details, see~\cite{BV2015}.
\end{remark}

The significance of Theorem~\ref{theo:multH} is that it shows that the multiplicity function $\lambda\mapsto m_K^{\CH}(\lambda)$ is a ``piecewise'' quasi-polynomial function supported on the Kirwan cone.
However, it is quite challenging to find an explicit decomposition of $C_K(\CH)$ into ``cones of quasi-polynomiality''.
Even when $K$ is a torus, this is the difficult problem of describing the decomposition of the cone $\Cone(\Psi)$ in so-called \emph{chambers} (see~\cite{BDV}).

To circumvent this difficulty, our iterated residues algorithm will produce, for a given input $\lambda^0\in C_K(\CH)$, a quasi-polynomial function that coincides with $m_K^\CH$ \emph{locally} in a closed cone $C$ containing $\lambda^0$, by deforming $\lambda^0$ in the direction of $C_K(\CH)$.
More precisely, we find a regular $\delta\in C_K(\CH)$ that is sufficiently small such that $\xi=\lambda^0+\delta$ is regular and the tope $\a=\a(\xi)$ determined by $\xi$ (Definition~\ref{def:admissible etc}) contains $\lambda^0$ in its closure.
Consider $C=\bigcap_{\sigma\in \mathcal{OS}(\Psi,\a)} \Cone(\sigma)$, which is an intersection of closed simplicial cones, all containing the tope~$\a$.
Then $\lambda^0\in C$, and our algorithm produces a quasi-polynomial that agrees with $m_K^\CH$ on $C\cap\Lambda_{K,\geq0}$.
(If $\lambda^0\not\in C_K(\CH)$ then $\xi\not\in C_K(\CH)$ and our method produces the zero quasi-polynomial, as desired.)

\subsection{Degrees and faces}\label{subsec:faces}

As a consequence of Theorem~\ref{theo:multH}, for any dominant weight $\lambda$ in $C_K(\CH)$, the dilated multiplicities $k\to m^{\CH}_K(k\lambda)$ are a quasi-polynomial function of $k$.
Thus it is of the form
$$m^{\CH}_K(k\lambda)=\sum_{i=0}^{N}c_i(k) k^i,$$
where $c_i(k)$ are periodic functions of $k$.
This formula is valid \emph{for all} $k\geq 0$ (so $c_0(0)=1$).
As before, the highest degree term for which this function is non-zero will be called the \emph{degree} of the quasi-polynomial function $m_K^{\CH}(k\lambda)$.
For any $\lambda$ contained in the relative interior of $C_K(\CH)$, this degree is $d$.
If $\lambda$ is in the boundary of $C_K(\CH)$, it is clear that the degree of the quasi-polynomial function $k\to m_K^{\CH}(k\lambda)$ cannot be larger than $d$, as it is the restriction of a quasi-polynomial of degree $d$.
Thus if the degree of $k\to m^\CH(k\lambda)$ is strictly smaller than $d$ then the corresponding point $\lambda$ is necessarily in the boundary of the Kirwan cone.

\begin{example}[Kronecker coefficients]\label{ex:kron degree}
  For the Kronecker case (Example~\ref{ex:kron}) $\dim C_K(\CH) \leq \sum_{j=1}^s (n_j-1) + 1$ (see~\cite{V-W} for necessary and sufficient conditions on when this is attained).
  It follows that the dilated Kronecker coefficients $k \mapsto g(k\nu_1,\ldots,k\nu_s)$ are quasi-polynomial functions of degree at most
$d = \prod_{j=1}^s n_j - \sum_{j=1}^s \frac {n_j(n_j-1)}2 - \sum_{j=1}^s (n_j - 1) - 1$.
\end{example}

Now consider a decomposition $C_K(\CH)=\bigcup_a \c_a$ into cones of quasi-po\-ly\-no\-mi\-a\-li\-ty.
We already remarked that the degree $d$ of the quasi-polynomial $p_{K,a}^{\CH}$ is the same for each $\c_a$.
However, the periods of the quasi-polynomials $p_{K,a}^{\CH}$ will in general be different in different cones $\c_a$.
On the example when $\CH=\C^6\otimes \C^3\otimes \C^2$ (see Section~\ref{exa:632}), we produce a cone $\c_a$ where $p_{K,a}^{\CH}$ is of degree $8$ and has the set of periods $\{1,2,3\}$, and a cone $\c_b$, for which $p_{K,b}^{\CH}$ is of degree $8$ and polynomial, that is, with period $\{1\}$.

When $F$ is a face of $C_K(\CH)$ then we can write $F=\bigcup_a (F\cap \c_a)$, where we restrict the union to those cones $\c_a$ for which $\dim (F\cap\c_a)=\dim F$ (i.e., to those closed cones $\c_a$ that contain a point $\xi_F$ in the relative interior of $F$).
We will say that these are the cones that are \emph{adjacent} to the face $F$.
The restriction to $F\cap \c_a$ of the function $m_K^{\CH}$ agrees with the restriction of the quasi-polynomial $p_{K,a}^{\CH}$.
Thus the multiplicity function $m_K^{\CH}$, restricted to a face $F$, is again a piecewise quasi-polynomial function.
Its degree drops, but if $F$ is a regular face (that is, a face that intersects the interior of the Weyl chamber), then the degree is the same on each cone $F\cap\c_a$ and can be computed by a formula analogous to the one in Theorem~\ref{theo:multH}.
In fact, the geometric Meinrenken-Sjamaar formula for multiplicities implies a reduction principle for the multiplicities on regular faces:
the function $m_K^{\CH}$ restricted to a regular face $F$ coincides with a multiplicity function $m_{K_0}^{\CH_0}$ for smaller data (see~\cite{BV2015} for details).
An example of this phenomenon in the context of Kronecker coefficients is given in Section~\ref{exa:632}.

\begin{remark}
  When $K$ is a torus, and $\c_a$ a cone of quasi-polynomiality adjacent to a facet $F$, then the quasi-polynomial $p_{K,a}^{\CH}$ vanishes on a certain number of affine hyperplanes parallel to the hyperplane generated by the facet $F$, leading to divisibility properties.
  We believe this is also the case more generally for the functions $p_{K,a}^{\CH}$, however, we do not have a precise guess.
  We give a striking example of this divisibility property in Section~\ref{exa:632}.
\end{remark}

\section{Branching rules}\label{sec:branch}

The results of the preceding section, in particular Theorem~\ref{thm:multT} and Lemma~\ref{lem:KandTmult}, can be turned into an algorithm for computing the multiplicities $m_K^\CH$ (see \cite[Theorem 59]{BV2015} for details).

As explained in the introduction, we will instead extend our considerations to the general branching problem for a pair of compact connected Lie groups.
Our main result, Theorem~\ref{theo:branchingsingular}, will give an explicit formula for the branching multiplicities that can be taylored for singular highest weights.
It forms the basis of our algorithm.

\subsection{Branching cones}

Consider a pair $K\subset G$ of compact connected Lie groups, with Lie algebras $\k,\g$ respectively.
Let $\pi:i\g^*\to i\k^*$ denote the corresponding projection.
As explained in Section~\ref{sec:setup}, we let $T_G$, $T_K$ denote maximal tori of $G$,$K$, with corresponding Cartan subalgebras $\t_\g$, $\t_\k$.
We may assume, and we do so, that $T_K\subset T_G$.
Given $\xi \in i\t^*_\g$, denote by $\overline {\xi}=\xi_{|_{i\t_\k}}$ its restriction to $i\t_\k$.
We may also choose compatible positive root systems on $K$, $G$: if $\lambda$ is dominant for $G$, then its restriction $\overline\lambda$ to $i\t_{\k}$ is dominant.
We denote the Weyl chambers by $i\t^*_{\g,\geq 0}$, $i\t^*_{\k,\geq 0}$, and the semigroups of dominant weights by $\Lambda_{G,\geq 0}$, $\Lambda_{K,\geq 0}$.

For $G\times K$, we denote by $\Lambda_{G,K,\geq 0}$ the sum $\Lambda_{G,\geq 0}\oplus\Lambda_{K,\geq 0},$ by $i\t^*_{\g,\k,\geq 0}$ the sum $i\t^*_{\g,\geq 0}\oplus i\t^*_{\k,\geq 0}$ of the positive Weyl chambers for $G,K$, and by $i\t^*_{\g,\k,> 0}$ its interior.
Define
$$V=\bigoplus_{\lambda\in \Lambda_{G,\geq 0}} V_\lambda^G\otimes V_{\lambda^*}^G.$$
So, under the action of $G\times K$,
$$V=\bigoplus_{\lambda,\mu} m_{G,K}(\lambda,\mu) V_\lambda^G\otimes V_{\mu^*}^K,$$
where $\lambda$ varies in $\Lambda_{G,\geq 0}$, and $\mu$ in $\Lambda_{K,\geq 0}$.
Here, $m_{G,K}(\lambda,\mu)$ is the multiplicity of the representation $V_\mu^K$ in the restriction of $V_\lambda^G$ to $K$.
The problem of \emph{branching rules} for $G,K$ is to determine these multiplicities.

We now define the \emph{branching cone} by
$$C_{G,K}=\Big\{ (\xi,\eta)\in i\t^*_{\g,\geq 0}\oplus i\t^*_{\k,\geq 0}, \eta \in \pi(G\cdot \xi)\Big\} \subseteq i\t^*_{\g,\k,\geq0}.$$
It is another instance of the theorems of Kirwan and Mumford that $C_{G,K}$ is a polyhedral cone, that the support of the function $m_{G,K}(\lambda,\mu)$ is contained in
 $C_{G,K}$ and that its asymptotic support is exactly the cone $C_{G,K}$ (cf.\ Proposition~\ref{prp:mumford}).

If $G=K$, the cone $C_{G,K}$ is just the diagonal $\{(\xi,\xi), \xi\in i\t^*_{\g,\geq 0}\}$ in $i\t^*_{\g,\g,\geq 0}$.
From now on, we will assume that \emph{no non-zero ideal of $\k$ is an ideal of $\g$}.
This condition excludes the preceding case and implies that the cone $C_{G,K}$ is solid~\cite{Duflo-pc}.
We say that a cone $\c\subset i\t_\k^*$ is solid if its interior is non-empty.
As anticipated after Theorem~\ref{theo:multH}, the Meinrenken-Sjamaar theorem also implies the following result:

\begin{theorem}[Meinrenken-Sjamaar, {\cite{M-S}}]\label{theo:mGK}
  There exists a cone decomposition $C_{G,K}=\bigcup_a \c_a$ into solid polyhedral cones $\c_a$ such that $m_{G,K}(\lambda,\mu)$ is given by a non-zero quasi-polynomial function on each $\c_a\cap (\Lambda_G\oplus \Lambda_K)$.
\end{theorem}

In particular, for any pair  $(\lambda,\mu)$ of dominant weights contained in $C_{G,K}$, the function $k\mapsto m_{G,K}(k\lambda,k\mu)$ is a quasi-polynomial and thus of the form
$m_{G,K}(k\lambda,k\mu)=\sum_{i=0}^{N}E_i(k) k^i,$
where the $E_i(k)$ are periodic functions of $k$. This formula is valid \emph{for all} $k\geq 0$; in particular, $E_0(0)=1$.

To describe the cone $C_{G,K}$ is difficult and has been the object of numerous works, notably by Berenstein-Sjamaar, Belkale-Kumar, Kumar, and Ressayre.
We refer to the survey article~\cite{Brion-bourbaki} for further detail.
The complete description of the multiplicity function $m_{G,K}$, and, in particular, the decomposition of $C_{G,K}$ into $\bigcup_a \c_a$ is even more so.
However, we will give an algorithm where, given as input $(\lambda,\mu)$, the output is the dilated function $k\mapsto m_{G,K}(k\lambda,k\mu)$.
In particular, we can  test if the point $(\lambda,\mu)$ is in the cone $C_{G,K}$ or not, according to whether the output is not zero or zero.

\subsection{A piecewise quasi-polynomial branching formula}\label{newsec:branch}

In principle, the function $m_{G,K}(\lambda,\mu)$ can be computed by the Heckman formula~\cite{Heck}.
However, in the cases that we are interested in, the Heckman formula is of formidably difficulty, but on the other hand the parameter $\lambda$ will often be quite singular.
In this section, we will obtain formulae for $m_{G,K}(\lambda,\mu)$ that are perhaps less beautiful but of much smaller complexity, taking avantage of the fact that $\lambda$ vanishes on a large number of coroots $H_\alpha$.
We will comment in Remark~\ref{heckversussingular} on the advantages of writing a specific formula for the singular case instead of using Kostant-Heckman branching theorem.

To start, we consider the list $\Psi$ of \emph{non-zero} restrictions of the roots $\Delta_\g^+$ to $i\t_\k$ and choose an order.
We say that $\Psi$ is the list of \emph{restricted roots} for the pair $\g,\k$.
As in Definition~\ref{def:admissible etc}, we write $\CA(\Psi)$ for the set of $\Psi$-admissible hyperplanes.

\begin{definition}\label{def:tope for F}
  Let $H\in\CA(\Psi)$ be an admissible hyperplane, with normal vector $X\in i\t_k$, so that $H=X^\perp$, and $w \in W_\g$ an element in the Weyl group of $G$.
  We define the hyperplane
  \[ H(w) = \left\{ (\xi,\nu) \in i\t_\g^* \oplus i\t_\k^*, \langle \overline{w(\xi)}, X \rangle - \langle \nu, X \rangle = 0 \right\}. \]
  Let $\CF$ denote the finite set of hyperplanes in $i\t_\g^*\oplus i\t_\k^*$ obtained by varying $H$ in $\CA(\Psi)$ and $w$ in the Weyl group of $G$.

  We say that $\tau$ is a \emph{tope for $\CF$} if it is a connected component of the complement of the union of the hyperplanes in $\CF$.
\end{definition}

Any tope $\tau$ for $\CF$ is an open polyhedral conic subset of $i\t_\g^* \oplus i\t_\k^*$.
By definition, if $(\xi,\nu)\in\tau$ then for any $\Psi$-admissible hyperplane $H=X^\perp$ and $w\in \CW_\g$, we have that
$\langle \overline{w(\xi)} - \nu,X\rangle \neq 0$.
Thus, for each $w\in \CW_\g$, the element $\overline{w(\xi)}-\nu$ is $\Psi$-regular and hence determines a $\Psi$-tope $\a(\overline{w(\xi)}-\nu)$.
This tope only depends on $w$ and $\tau$, but not on the choice of $(\xi,\nu)\in\tau$, so the following definition makes sense:

\begin{definition}\label{def:induced tope}
  Let $\tau$ be a tope for $\CF$ and $w\in \CW_\g$.
  We denote by $\a_w^\tau$ the unique $\Psi$-tope equal to $\a(\overline{w(\xi)}-\nu)$ for all $(\xi,\nu)\in\tau$.
\end{definition}

The facets of the cones $\c_a$ in Theorem~\ref{theo:mGK} generate hyperplanes that belong to the family $\CF$ defined above, as follows from the description of the Duistermaat-Heckman measure~\cite{Heck}.
Thus, if $\c_a$ is such a cone and $\tau$ a tope for $\CF$, then $\tau\cap i\t^*_{\g,\k,\geq 0}$ is either fully contained in $\c_a$ or disjoint from $\c_a$.
In particular, each closed cone $\c_a$ is the union of the closures of the sets $\tau\cap i\t^*_{\g,\k,\geq 0}$ over the $\tau$ such that $\tau\cap \c_a$ is non-empty.
In general, there might be several topes $\tau$ needed to obtain $\c_a$.

\medskip

We now extend these definitions so that we can later take advantage of singular highest weights $\lambda$.
Thus fix a subset $\Sigma$ of the simple roots in $\Delta_\g^+$.
Let $i \t_{\g,\Sigma}^*$ be the set of the elements $\xi\in i\t_\g^*$ such that $\la \xi,H_\alpha\ra=0$ for all $\alpha\in \Sigma$.
We define correspondingly $\t^*_{\g,\k,\Sigma, \geq 0}=i\t^*_{\g,\Sigma} \oplus i\t^*_\k$,
$\Lambda^{\Sigma}_{G}=\Lambda_G \cap i \t_{\g,\Sigma}^*$ (a lattice in $i\t_{\g, \Sigma}^*$), $\Lambda_{G,K,\geq 0}^\Sigma=\Lambda_{G,K,\geq 0}\cap i \t^*_{\g, \k,\Sigma}$, etc.
Lastly, we define the corresponding branching cone by
$$C^\Sigma_{G,K}=\{(\xi,\nu)\in C_{G,K}, \xi\in i \t_{\g,\Sigma}^*\} \subseteq \t^*_{\g,\k,\Sigma, \geq 0}.$$
If $\Sigma$ is empty, then $C^\Sigma_{G,K}=C_{G,K}$.
Otherwise, $(\xi,\nu)\in C^\Sigma_{G,K}$ if $\nu$ belongs to the projection to $i\k^*$ of the singular orbit $G\xi$.
Thus the cone $C^\Sigma_{G,K}$ is contained in the cone $C_{G,K}$ and is in its boundary if $\Sigma$ is non-empty.
We would like to derive formulae for the function $m_{G,K}$ restricted to this cone $C^\Sigma_{G,K}$.

Note that the cone $C^\Sigma_{G,K}$ is solid in $i \t_{\g, \Sigma}^*$ if and only if there exists $\xi\in i \t_{\g,\Sigma}^*$ such that the projection to $i\k^*$ of the singular orbit
$G\xi$ has a non-zero interior in $i\k^*$.
In other words, $C^\Sigma_{G,K}$ is solid if and only if the Kirwan polytope $\pi(G\xi)\cap i\t^*_\k$ for some $\xi\in i\t_{\g,\Sigma}^*$ is solid.

The following example will later be used in the algorithm to determine Kronecker coefficients.

\begin{example}[{\cite{V-W}}]\label{ex:kron solid}
  Let $s\geq3$ and $n_2,\dots,n_s\geq2$, and consider the embedding of $K=SU(n_2)\times\dots\times SU(n_s)$ into $G=U(M)$, where $M=n_2 \cdots n_s$.
  Let $\lambda\in P\Lambda_{U(M),\geq0}$ be a dominant weight with two or more non-zero coordinates.
  Then $\pi(G\lambda)$ has interior in $i\k^*$.
  It follows that the cones $C^\Sigma_{G,K}$ are solid whenever $\Lambda^{\Sigma}_{U(M),\geq0}$ contains such a $\lambda$.
\end{example}

The following definition extends Definition~\ref{def:tope for F}.

\begin{definition}\label{def:tope for F Sigma}
  Let $\CF_\Sigma$ denote the set of hyperplanes in $i\t^*_{\g, \Sigma}\oplus i\t^*_\k$ of the form
  \[ \left\{ (\xi,\nu) \in i\t_{\g,\Sigma}^* \oplus i\t_\k^*, \langle \overline{w(\xi)}, X \rangle - \langle \nu, X \rangle = 0 \right\}, \]
  obtained by varying $H=X^\perp$ in $\CA(\Psi)$ and $w$ in the Weyl group $\CW_\g$.
  Thus each element in $\CF_\Sigma$ is an intersection of a hyperplane $H(w)$ with $i\t^*_{\g,\Sigma} \oplus i\t^*_\k$.

  We say that $\tau_\Sigma$ is a \emph{tope for $\CF_\Sigma$} if it is a connected component of the complement of the union of the hyperplanes in $\CF_\Sigma$.
\end{definition}

Note that any tope $\tau_\Sigma$ for $\CF_\Sigma$ is contained in a unique tope $\tau$ for $\CF$.

If $\tau_\Sigma$ is a tope for $\CF_\Sigma$, then $\tau_\Sigma\cap C^\Sigma_{G,K}$ is empty when $C^\Sigma_{G,K}$ is not solid.
Otherwise, if the cone $ C^\Sigma_{G,K}$ is solid, then it is the union of the closures of the sets $\tau_\Sigma\cap \t^*_{\g,\k,\Sigma,\geq0}$ contained in $C^\Sigma_{G,K}$.

\medskip

Our goal in the following now is to give quasi-polynomial formulas for the function $m_{G,K}$ restricted to $\tau_\Sigma\cap C_{G,K}^\Sigma$, where $\tau_\Sigma$ is a tope for $\CF_\Sigma$.
Similarly to in Section~\ref{sec:multiplicities}, we first discuss the multiplicities for the maximal torus $T_G$.
Throughout the following \emph{we will assume that $\Sigma$ is a subset of the simple roots in $\Delta_\g^+$ such that $C^\Sigma_{G,K}$ is solid}.
Let $\lp$ be the corresponding Levi subalgebra of $\g$ with simple roots $\Sigma$.
Denote by $\Delta_\lp^+$ its positive root system and by $\CW_{\lp}$ its Weyl group.
Lastly, define $\Delta_\uu=\Delta_\g^+ \setminus \Delta_\l^+$.
Then we have the following character formula:

\begin{lemma}\label{lem:restricted weyl}
  For any $\lambda\in\Lambda^\Sigma_{G,\geq 0}$, the character of the irreducible representation $V_\lambda^G$ is given by
  $${\chi_{\lambda}}_{|T_G}=\sum_{[w] \in \CW_\g/\CW_{\lp}}\frac{e^{w(\lambda)}}{\prod_{\alpha \in \Delta_{\uu}}(1-e^{-w(\alpha)})}.$$
\end{lemma}

\begin{remark}
  The formula in Lemma~\ref{lem:restricted weyl} is a special case of the Atiyah-Bott fixed point formula, and easily obtained from the Weyl character formula.
  When $\lambda$ is regular, $\Delta_\uu=\Delta^+_\g$, it is exactly the Weyl character formula.
  When $\lambda$ is singular, this sum is over an eventually much smaller set of elements $w$, and the summands are simpler functions.
  The extreme case is $\lambda=0$, with just one term, equal to $1$.
\end{remark}

\subsubsection{The ``regular'' case}\label{regular}

We start by considering the case where all roots in $\Delta_\g$ have non-zero restriction to $i\t_\k$.
The general situation will be treated in Section~\ref{notregular}.

In this case, the space $i\t_\k$ contains a regular (with respect to $\Delta_\k$) element $Y$, which is regular also for $\Delta_\g$.
We use this element to define compatible positive root systems $\Delta^+_\g$ and $\Delta^+_\k$ as follows:
$$\Delta^+_\g=\{ \alpha \in \Delta_\g, \alpha (Y)>0\}, \quad \Delta^+_\k=\{ \alpha \in \Delta_\k, \alpha (Y)>0\}.$$
Thus the list $\Psi$ consists of the restrictions of $\Delta_\g^+$, repeated according to their multiplicities (we implicitly choose an order):
$$\Psi=[\overline{\alpha}, \alpha \in \Delta^+_\g].$$
The list $\Psi$ contains $\Delta_\k^+$.
By construction, all elements $\psi$ in $\Psi$ satisfy $\langle \psi, Y\rangle >0$.

\begin{example}\label{ex:kron restri roots}
  Let $G=U(n_2n_3)$ and $K=SU(n_2)\times SU(n_3)$.
  We consider $\t_\g$ the Cartan subalgebra of $\g$ given by the diagonal matrices and $\t_\k=\t_2\times\t_3$, the Cartan subalgebras of $\k$ given by the corresponding diagonal matrices, with trace zero.
  The embedding of $K$ in $G$ leads to the embedding of $i\t_2\times i\t_3 \rightarrow i\t$ given by
  \begin{equation}\label{eq:kron embedding}
  \begin{aligned}
    &(a_1,\ldots,a_{n_2})\times(b_1,\ldots,b_{n_3}) \rightarrow (a_1+b_1, a_2+b_1,\ldots,a_{n_2}+b_1, \\
    &\qquad a_1+b_2, a_2+b_2, \ldots, a_{n_2}+b_2, \;\ldots\;, a_{1}+b_{n_3},\ldots, a_{n_2}+ b_{n_3}).
  \end{aligned}
  \end{equation}
  The list of restricted positive roots is
  \[ \Psi=[(a_i-a_j+b_k-b_\ell)], \]
  where $1\leq i<j\leq n_2$ and $1\leq k<\ell\leq n_3$, so all restricted roots are non-zero.
  We take the lexicographic order. 
  It is easy to see that the standard positive root systems $\Delta_\k^+$ and $\Delta_\g^+$ introduced at the beginning of Section~\ref{sec:setup} are compatible.
  Explicitly, we may choose
  \begin{equation*}
  \begin{aligned}
    Y=\bigl(&(n_2-1, n_2-3, \ldots,-n_2+1), \\
    &((n_3-1)n_2, (n_3-3)n_2, \ldots,(-n_3+1)n_2)\bigr].
  \end{aligned}
  \end{equation*}
  For example, for $n_2=3, n_3=2$, $Y=((2,0,-2), (3,-3))$.
  Then the embedded element is $(5,3,1,-1,-3,-5)$.
  Both are regular dominant with respect to the standard positive root systems for $K$ and $G$, respectively.

  An analogous construction works more generally for $G=U(n_2\cdots n_s)$ and $K=SU(n_2)\times\dots\times SU(n_s)$, $s\geq3$.
\end{example}

Using Lemma~\ref{lem:restricted weyl}, we can write the restriction of the character $\chi_\lambda$ to $T_K$ as a sum of meromorphic functions, indexed by cosets $[w]\in\CW_\g/\CW_\lp$:
\begin{equation}\label{eq:restricted character}
  {\chi_{\lambda}}_{|T_K}=\sum_{[w] \in \CW_\g/\CW_{\lp}}\frac{e^{{\overline{w(\lambda)}}}}{\prod_{\alpha \in \Delta_{\uu}}(1-e^{-{\overline{w(\alpha)}}})}.
\end{equation}
Note that none of the denominators are identically zero, since by assumption the restricted roots $\overline{w(\alpha)}$ are all non-zero.

To compute $m_{G,K}(\lambda,\mu)$ for $\Lambda^\Sigma_{G,K,\geq 0}$ by iterated residues, we consider the following function of $z\in (\t_\k)_\C$:
\begin{equation}\label{eqSSigma}
S^{\Sigma, w}_{\lambda,\mu}(z) = \left( \prod_{\beta\in \Delta_\k^+}(1-e^{-\la \beta,z\ra}) \right) \frac{e^{\la{\overline{w(\lambda)}-\mu,z\ra}}}{\prod_{\alpha \in \Delta_{\uu}}(1-e^{-{\la\overline{w(\alpha)},z\ra}})}.
\end{equation}
Now let $q=q(\Psi)$ denote the index of $\Psi$ in $L=\Lambda_K$, the weight lattice of $K$.
Recall that if $\tau_\Sigma$ is a tope for $\CF_\Sigma$ then it is contained in a unique tope $\tau$ for $\CF$, which in turn defines $\Psi$-topes $\a_w^{\tau}$ for any $w\in\CW_\g$ (Definition~\ref{def:induced tope}).
The tope $\a_w^\tau$ in fact only depends on the coset $[w] \in \CW_\g / \CW_\lp$.
Lastly, recall that ${\mathcal{OS}}(\Psi,\a^\tau_w)$ is the set of OS bases adapted to the tope $a_w^{\tau}$ (Def. \ref{OS}).

\begin{definition}\label{def:branchfunct}
  Let $\tau_\Sigma$ be a tope for $\CF_\Sigma$ and $\tau$ the unique tope for $\CF$ that contains $\tau_\Sigma$.
  Define
  $$p_\tau^\Sigma(\lambda,\mu)=\sum_{[w] \in \CW_\g/ \CW_\lp} \sum_{\gamma\in \Gamma_K/q\Gamma_K} \sum_{\overrightarrow\sigma \in \mathcal{OS}(\Psi,\a^\tau_w) } \Res_{\overrightarrow\sigma} S^{\Sigma,w}_{\lambda,\mu}(z+\frac{2i\pi\gamma}{q}).$$
\end{definition}

\begin{remark}\label{rem:branchfunct quasipol}
  As in Remark~\ref{rem:szenes vergne quasipol}, it is clear that $p_\tau^\Sigma(\lambda,\mu)$ is a quasi-polynomial function on $\Lambda^\Sigma_G \oplus \Lambda_K$.
\end{remark}

We now state our branching formula:

\begin{theorem}\label{theo:branch}
  Assume that $C^\Sigma_{G,K}$ is solid and that all roots in $\Delta_\g$ have non-zero restriction to $i\t_\k$.
  Let $\tau_\Sigma \subset i\t^*_{\g,\k,\Sigma}$ be a tope for $\CF_\Sigma$, and $\tau$ the tope for $\CF$ that contains $\tau_\Sigma$.
  Let $(\lambda,\mu)\in \overline \tau_\Sigma\cap \Lambda^\Sigma_{G,K,\geq 0}$.
  Then:
  \begin{enumerate}
  \item[(i)] If $(\lambda,\mu)\notin C^\Sigma_{G,K}$, then
    $$m_{G,K}(\lambda, \mu)=p^\Sigma_\tau(\lambda,\mu)=0.$$
  \item[(ii)] If $(\lambda,\mu)\in C^\Sigma_{G,K}$ \emph{and} the tope $\tau_\Sigma$ intersects $C^\Sigma_{G,K}$ (hence, $\overline{\tau_\Sigma}\cap i\t^*_{\g,\k,\Sigma,\geq0} \subseteq C^\Sigma_{G,K}$), then
    $$m_{G,K}(\lambda, \mu)=p^\Sigma_\tau(\lambda,\mu).$$
\end{enumerate}
\end{theorem}
\begin{proof}
  The set $\tau_\Sigma\cap C_{G,K}^{\Sigma}$ is contained in $\tau \cap C_{G,K}$.
  We know from Theorem~\ref{theo:mGK} and the discussion below Definition~\ref{def:induced tope} that, on $\overline\tau \cap \Lambda_{G,K,_\geq 0}$, the function $m_{G,K}$ is given by a quasi-polynomial formula. A fortiori, its restriction to $\overline {\tau_\Sigma} \cap \Lambda^{\Sigma}_{G,K, \geq 0}$ is likewise given by a quasi-polynomial.
  Thus it is sufficient to prove that $m_{G,K}(\lambda,\mu)=p^\Sigma_\tau(\lambda,\mu)$ when $(\lambda,\mu)\in\tau_\Sigma \cap \Lambda^\Sigma_{G,K,\geq0}$ is sufficiently far away from all the walls belonging to $\CF_\Sigma$ (see the uniqueness result in Lemma~\ref{lem:equaquasipoly} below).

  We start with the restricted character ${\chi_\lambda}_{|T_K}$ as given by Eq.~\eqref{eq:restricted character}.
  For each $w\in\CW_\g$, we can rewrite the formula by polarizing the linear forms $\overline{w(\alpha)}$ in the denominator using our regular element $Y$:
  if $\la\overline{w(\alpha)},Y\ra<0$, we replace $\overline{w(\alpha)}$ by its opposite; we then make use of the identity $\frac{1}{1-e^{-\beta}}=-\frac{e^{\beta}}{1-e^{\beta}}$.
  More precisely, we define $\Psi_{w,\uu}=\Psi^1_{w,\uu} \cup \Psi^2_{w,\uu}$ with
  \begin{align*}
    \Psi^1_{w,\uu}&=\{\hphantom{-}\overline{w(\alpha)}, \ \alpha \in \Delta_\uu, \ \la\overline{w(\alpha)},Y\ra >0\}, \\
    \Psi^2_{w,\uu}&=\{-\overline{w(\alpha)}, \ \alpha \in \Delta_\uu,  \ \la\overline{w(\alpha)},Y\ra <0\}.
  \end{align*}
  The elements in $\Psi_{w,\uu}$ are positive on $Y$, so $\Psi_{w,\uu} \subseteq \Psi$.
  (However, note that $\Psi_{w,\uu}$ depends on $w$ and may not contain $\Delta_\k^+$.)
  We also define $s^\Sigma_w = \lvert\Psi^2_{w,\uu}\rvert$ and $e^{g^\Sigma_w} = \prod_{\psi\in\Psi^2_{w,\uu}} e^{-\psi}$.
  We can thus rewrite Eq.~\eqref{eq:restricted character} as
  \begin{align*}
    {\chi_\lambda}_{|T_K}
  &= \sum_{[w] \in \CW_\g/\CW_\lp} \frac{e^{\overline{w(\lambda)}}} {\prod_{\psi\in\Psi^1_{w,\uu}} (1-e^{-\psi})} \frac{(-1)^{s^\Sigma_w} e^{g^\Sigma_w}}{\prod_{\psi  \in \Psi^2_{w,\uu}} (1-e^{-\psi})} \\
  &= \sum_{[w] \in \CW_\g/\CW_\lp} \frac{e^{\overline{w(\lambda)}} (-1)^{s^\Sigma_w} e^{g^\Sigma_w}}{\prod_{\psi \in  \Psi_{w,\uu}}(1-e^{-\psi})}.
  \end{align*}
  We obtain the following expression for the restricted character:
  \[ {\chi_\lambda}_{|T_K} = \sum_{[w] \in \CW_\g/\CW_\lp} (-1)^{s^\Sigma_w} \sum_{\nu\in\Lambda_K} \CP_{\Psi_{w,\uu}}(\overline{w(\lambda)}+g_w^\Sigma-\nu) \, e^\nu, \]
  where $\CP_{\Psi_{w,\uu}}$ is the partition function determined by the restricted roots~$\Psi_{w,\uu}$ (cf.\ Eq.~\eqref{eq:geometric}).
  As a consequence:
  \begin{equation}\label{mult}
    m_{G,T_K}(\lambda,\nu)=\sum_{[w] \in \CW_\g/\CW_\lp} (-1)^{s_w^\Sigma} \CP_{\Psi_{w,\uu}}(\overline{w(\lambda)}+g_w^\Sigma-\nu)
  \end{equation}
  When $K$ is the maximal torus $T_G$ of $G$ and $\lambda$ regular, Formula~\eqref{mult} is the Kostant multiplicity formula for a weight~\cite{Kos}.
  The formula above is obtained by the same method (using that the linear forms in the denominator are polarized).

  Let us now use Lemma~\ref{lem:KandTmult}, $$m_{G,K}(\lambda,\mu) = \sum_{\tilde w \in \CW_\k} \epsilon(\tilde w) m_{G,T_K}(\lambda, \mu+\rho_\k-\tilde w(\rho_\k)).$$
  Inserting Eq.~\eqref{mult}, we obtain for $(\lambda,\mu)\in \Lambda_{G,K,\geq 0}$ that
  \[ m_{G,K}(\lambda, \mu) = \sum_{\tilde w \in \CW_\k} \epsilon(\tilde w) \!\!\!\! \sum_{[w] \in \CW_\g/\CW_\lp} \!\!\!\! (-1)^{s^\Sigma_w} \CP_{\Psi_{w,\uu}}\bigl(\overline{w(\lambda)} + g^\Sigma_w - (\mu+\rho_k-\tilde{w}(\rho_\k))\bigr). \]
  Observe that, if $\lambda$ is regular, then we may rewrite this expression as a sum of partitions functions for $\Psi\backslash \Delta_\k^+$, obtaining the Heckman formula~\cite{Heck} -- but we will not use this fact. 
  The point $(\lambda,\mu)$ being in $\tau_\Sigma$, it follows that $\overline{w(\lambda)}-\mu$ is in the $\Psi$-tope $\a_w^\tau$.
  Since $\Psi_{w,\uu}\subseteq\Psi$, $\a_w^\tau$ in turn is contained in a unique $\Psi_{w,\uu}$-tope, which we denote by the same symbol.
  We can assume that $(\lambda,\mu)$ is sufficiently far from all walls, so that all the translates $\overline{w(\lambda)} + g^\Sigma_w - (\mu+\rho_k-\tilde{w}(\rho_\k))$ are also in $\a_w^\tau$.
  Now use Theorem~\ref{thm:multT} to express the values of the partition function by iterated residues:
  \begin{align*}
    &m_{G,K}(\lambda,\mu) = \sum_{\tilde w \in \CW_\k} \epsilon(\tilde w) \!\!\!\! \sum_{[w] \in \CW_\g/\CW_\lp} \!\!\!\! (-1)^{s^\Sigma_w} \\
    &\quad\sum_{\gamma\in \Gamma_K/q\Gamma_K} \sum_{{\overrightarrow{\sigma}} \in \mathcal{OS}(\Psi_{w,\uu},\a^\tau_w)} \!\!\!\!
    \Res_{\overrightarrow\sigma} \frac{e^{\la \overline{w(\lambda)}+g_w^\Sigma-(\mu+\rho_\k-\tilde{w}(\rho_\k)), z+\frac{2i\pi}{q} \gamma\ra}} {\prod_{\psi \in \Psi_{w,\uu}} (1-e^{-\la \psi,z+\frac{2i\pi}{q} \gamma\ra})}.
  \end{align*}
  As $\Psi_{w,\uu} \subseteq \Psi$, the index $q(\Psi)$ is a multiple of $q(\Psi_{w,\uu})$ for all $w$.
  We may thus use $q=q(\Psi)$ for computing all terms above (see Remark~\ref{rem:szenes vergne quasipol}).

  Reverting the polarization process, we may rewrite
  $$(-1)^{s_w^\Sigma}\frac{e^{\la \overline{w(\lambda)}+g_w^\Sigma,z+\frac{2i\pi}{q} \gamma\ra}} {\prod_{\psi \in \Psi_{w,\uu}} (1-e^{-\la \psi,z+\frac{2i\pi}{q} \gamma\ra})}=\frac{e^{\la \overline{w(\lambda)}, z+\frac{2i\pi}{q} \gamma\ra}} {\prod_{\alpha \in \Delta_\uu} (1-e^{-\la \overline{w\alpha},z+\frac{2i\pi}{q} \gamma\ra})}.$$
  So, remembering that $\prod_{\beta\in \Delta_\k^+}(1-e^{-\beta})=\sum_{\tilde{w}} \epsilon(\tilde w) e^{-\rho_\k+\tilde{w}(\rho_\k)}$, we obtain
  \begin{align*}
    m_{G,K}(\lambda,\mu) = \!\!\!\! \sum_{[w] \in \CW_\g/ \CW_\lp} \sum_{\gamma\in \Gamma_K/q\Gamma_K} \sum_{\overrightarrow\sigma \in \mathcal{OS}(\Psi_{w,\uu},\a^\tau_w) }
    \!\!\!\! \!\!\!\!\Res_{\overrightarrow\sigma} S^{\Sigma,w}_{\lambda,\mu}(z+\frac{2i\pi\gamma}{q}).
  \end{align*}
  Lastly, we may replace the sum over the bases in $\mathcal{OS}(\Psi_{w,\uu},\a^\tau_w)$ by a sum over the bases in $\mathcal{OS}(\Psi,\a^\tau_w)$, since $\Psi_{w,\uu}\subseteq\Psi$ (see Lemma~\ref{lem:OS bases} below).
  Thus we have shown that
  $$m_{G,K}(\lambda, \mu)=p_\tau(\lambda,\mu)$$
  when $(\lambda,\mu)\in \tau_\Sigma\cap\Lambda_{G,K,\geq 0}^\Sigma$ is sufficiently far away from any of the walls in $\CF_\Sigma$.
  Both the left and the right-hand side are quasi-polynomials (Theorem~\ref{theo:mGK} and Remark~\ref{rem:branchfunct quasipol}).
  Thus the theorem now follows from Lemma~\ref{lem:equaquasipoly} below, which asserts that any quasi-polynomial function is uniquely determined by its values on a sufficiently large subset.
\end{proof}

\begin{lemma}\label{lem:OS bases}
  Let $\Psi_0\subseteq\Psi$ be a sublist, $\xi$ a $\Psi$-regular element, and $\a_0\supseteq\a$ the topes determined by $\xi$.
  Let $f(z)=g(z)/h(z)$, with $g(z)$ holomorphic near $z=0$ and $h(z)=\prod_{\psi\in\Psi_0} \la \psi, z \ra$.
  Then:
  \[ \sum_{\overrightarrow\sigma\in\mathcal{OS}(\Psi,\a)} \Res_{\overrightarrow\sigma}(f)
  = \sum_{\overrightarrow\sigma\in\mathcal{OS}(\Psi_0,\a_0)} \Res_{\overrightarrow\sigma}(f) \]
  In particular, if $\Psi_0$ does not generate $E$ then both sides are zero.
\end{lemma}
\begin{proof}
  Taking the Taylor series of $g(z)$ at $z=0$, we may assume that $g$ is polynomial.
  If $\Psi_0$ does not generate $E$ then both residues are zero, since they depend only on the homogeneous component of $g/h$ of degree $-\!\dim E$.
  Thus we may assume that $f$ is homogeneous of degree $-\!\dim E$.

  If $\Psi_0$ generates $E$ then both sides compute the same object:
  the Jeffrey-Kirwan residue of $f$ on $\a$ (see~\cite[Proposition 8.11]{JeffreyKirwan1993}).
  Let us recall its definition.
  Consider the unique generalized function $\theta$ on $E$, supported on the pointed cone $\Cone(\Psi_0)\subseteq\Cone(\Psi)$, such that the Fourier transform $\int_E e^{i \la\xi,z\ra} \theta(\xi) d\xi$ coincides with $f$ on the open subset $\{ z\in E^* : \prod_{\psi\in\Psi_0} \la\psi,z\ra\neq0 \}$ of $E^*$.
  The function $\theta(\xi)$ is constant on the $\Psi_0$-tope $\a_0$.
  Its value on $\a_0$ is by definition the Jeffrey-Kirwan residue of $f$ on $\a_0$.
  Now the left-hand side computes the value of $\theta$ at an arbitrary point $\xi\in\a$, while the right-hand side computes its value at an arbitrary point $\xi\in\a_0$.
  Since $\a\subseteq\a_0$, both sides coincide.
\end{proof}

\begin{lemma}\label{lem:equaquasipoly}
  Let $p_1,p_2$ be two quasi-polynomial functions on a lattice $L$.
  If there exists a  open cone $\tau$ and $s\in L$ such that $p_1,p_2$ agree on a translate $(s+\tau)\cap L$ of $\tau$, then $p_1=p_2$.
\end{lemma}

The proof of Lemma~\ref{lem:equaquasipoly} is left to the reader.

\begin{remark}\label{rem:opti}
  While in Definition~\ref{def:branchfunct} we sum over all OS bases of $\Psi$, the proof of Theorem~\ref{theo:branch} shows that, for each $[w]\in\CW_\g/\CW_\lp$, it suffices to sum only over the OS bases of $\Psi_{w,\uu}$.

  In fact, for any given $[w]\in\CW_\g/\CW_\lp$ and $\gamma\in\Gamma_K/q\Gamma_K$, we may use the OS bases of the system
  \[ \Psi_{w,\gamma,\uu} = \{ \psi \in \Psi_{w,\uu} \setminus \Delta^+_\k, e^{\frac{2i\pi}q\la\psi,\gamma\ra}=1 \} \]
  (when $L$ and $L'$ are lists, we write $L \setminus L'$ for the difference of lists, i.e., we remove elements according to their multiplicity).
  That is, we may use the formula
  \[ p_\tau^\Sigma(\lambda,\mu)=\sum_{[w] \in \CW_\g/ \CW_\lp} \sum_{\gamma\in \Gamma_K/q\Gamma_K} \sum_{\overrightarrow\sigma \in \mathcal{OS}(\Psi_{w,\gamma,\uu},\a^\tau_w) } \Res_{\overrightarrow\sigma} S^{\Sigma,w}_{\lambda,\mu}(z+\frac{2i\pi\gamma}{q}). \]
  This follows from Lemma~\ref{lem:OS bases}, since the only poles of the function $z \mapsto S^{\Sigma,w}_{\lambda,\mu}(z+2i\pi/q)$ near $z=0$ come from the linear forms in $\Psi_{w,\gamma,\uu} \setminus \Delta^+_\k$. 
  Likewise, we may choose $q$ as the index of $\Psi\setminus\Delta_\k^+$ (or even of $\Psi_{w,\uu}\setminus\Delta_\k^+$).
  We use these optimizations in our algorithm for computing Kronecker coefficients (see Appendix~\ref{app:thealgoKro}).
\end{remark}

\begin{remark}
  An obvious necessary criterion for $\mathcal{OS}(\Psi_{w,\gamma,\uu}, \a^\tau_w)$ to be non-empty is that $\overline{w(\lambda)}-\mu$ is contained in the cone generated by the $\Psi$. 
  The set of such $w$ is called the set of valid permutations for $(\lambda,\mu)$, defined by Cochet in~\cite{C2}.
  See also the notion of Weyl alternative sets in~\cite{PH}.
\end{remark}

\subsubsection{The general case}\label{notregular}

We now explain how to deal with the general case.
As before, we start with the character formula (Lemma~\ref{lem:restricted weyl}):
$$\chi_{{\lambda}|_{T_G}}=\sum_{[w] \in \CW_\g/\CW_{\lp}}\frac{e^{w(\lambda)}}{\prod_{\alpha \in \Delta_{\uu}}(1-e^{-w(\alpha)})}.$$
But now we can not longer directly restrict the formula to $T_K$, since the denominator could vanish identically on $T_K$.
So we instead compute a limit formula as follows.

Choose $Y_1\in i\t_\g$ such that $\langle w(\alpha),Y_1\rangle \neq 0$ for all $\alpha \in \Delta_\uu$ and $[w]\in\CW_\g/\CW_{\lp}$ for which the restriction $\overline{w(\alpha)}$ of $w(\alpha)$ to $i\t_\k$ is zero.
Let $z\in (\t_\k)_\C$ and $[w]\in \CW_\g/\CW_{\lp}.$
For $\epsilon$ small, consider the expression
$e^{\langle \overline{w(\lambda)}, z\rangle}e^{\langle w(\lambda), \epsilon Y_1\rangle} / \prod_{\alpha \in \Delta_{\uu}}(1-e^{-\langle w(\alpha),z+\epsilon Y_1\rangle})$
and define
\[ \Theta_\epsilon(z,Y_1)=\frac{e^{\langle w(\lambda), \epsilon Y_1\rangle}} {\prod_{\alpha \in \Delta_{\uu}}(1-e^{-\langle w(\alpha),z+\epsilon Y_1\rangle})}. \]
The function $\epsilon\mapsto\Theta_\epsilon(z,Y_1)$ has a pole at $\epsilon=0$ of order $p=\lvert\{\alpha \in \Delta_\uu, \langle w(\alpha),z\rangle=0\}\rvert.$
Given a Laurent series $\sum_{i\geq-p} c_i\epsilon^i$ at $\epsilon=0$, we say that $c_0$ is its constant term.
Thus define $F_w(z,Y_1)$ to be the constant term of the Laurent expansion of $\epsilon\mapsto\Theta_\epsilon(z,Y_1)$ at $\epsilon=0$, and let
$$G_w(z,Y_1)=e^{\langle \overline{w(\lambda)}, z\rangle}F_w(z,Y_1).$$

When the order $p$ of the pole is $0$, as in Section~\ref{regular}, then
$$G_w(z,Y_1) = \frac{e^{\la{\overline{w(\lambda)}},z\ra}}{\prod_{\alpha \in \Delta_{\uu}}(1-e^{-{\la \overline{w(\alpha)}},z\ra})},$$
that is, $G_w(z,Y_1)$ is a summand in Eq.~\eqref{eq:restricted character}.
Thus, in this case, $\sum_{w \in \CW_\g/\CW_\lp} G_w(z,Y_1)$ does not depend on the choice of $Y_1$ and is equal to the restricted character $\chi_{\lambda}(\exp(z)).$

In general, $G_w(z,Y_1)$ is of the form $P/Q$ where $P$ is a sum of exponentials and $Q$ is a product of the form $\prod_{\psi \in \Psi} ( 1-e^ {-\psi})^{n_\psi}$, where $\Psi$ is as before the list of non-zero restricted roots.
Thus the restricted character ${\chi_{\lambda}}_{|T_K}$ can again be expressed as a sum of partition functions associated to lists of elements belonging to $\Psi$ (with possibly higher multiplicities).
More precisely, define
\begin{equation}\label{eq:new S}
\begin{aligned}
S_{\lambda,\mu}^{\Sigma,w}(z) &= \left( \prod_{\beta \in \Delta_{\k}}(1-e^{-\langle \beta, z\rangle}) \right) e^{\langle \overline{w(\lambda)}-\mu, z\ra} F_w(z,Y_1) \\
&= \left( \prod_{\beta \in \Delta_{\k}}(1-e^{-\langle \beta, z\rangle}) \right) e^{-\la\mu, z\ra} G_w(z,Y_1).
\end{aligned}
\end{equation}
In the case when the restriction to $i\t_\k$ of all $\alpha\in \Delta_\g$ is non-zero, the function $S_{\lambda,\mu}^{\Sigma, w}(z)$ is indeed equal to the function defined in Eq.~(\ref{eqSSigma}).
In the general case, the function $S_{\lambda,\mu}^{\Sigma,w}(z)$ depends on our choice of $Y_1$.
We leave this choice implicit.

By arguing as in the proof of Theorem~\ref{theo:branch}, we obtain the following result:

\begin{theorem}\label{theo:branchingsingular}
  Assume that $C^\Sigma_{G,K}$ is solid.
  Let $\tau_\Sigma \subset i\t^*_{\g,\k,\Sigma}$ be a tope for $\CF_\Sigma$, and $\tau$ the tope for $\CF$ that contains $\tau_\Sigma$.
  Let $(\lambda,\mu)\in \overline \tau_\Sigma\cap \Lambda^\Sigma_{G,K,\geq 0}$.
  Then:
  \begin{enumerate}
  \item[(i)] If $(\lambda,\mu)\notin C^\Sigma_{G,K}$, then
    $$m_{G,K}(\lambda, \mu)=p^\Sigma_\tau(\lambda,\mu)=0.$$
  \item[(ii)] If $(\lambda,\mu)\in C^\Sigma_{G,K}$ \emph{and} the tope $\tau_\Sigma$ intersects $C^\Sigma_{G,K}$ (hence, $\overline{\tau_\Sigma}\cap i\t^*_{\g,\k,\Sigma,\geq0} \subseteq C^\Sigma_{G,K}$), then
    $$m_{G,K}(\lambda, \mu)=p^\Sigma_\tau(\lambda,\mu).$$
  \end{enumerate}
  Here, $p_\tau^\Sigma(\lambda,\mu)$ is defined as in Definition~\ref{def:branchfunct} but using the more general definition of $S_{\lambda,\mu}^{\Sigma,w}$ in Eq.~\eqref{eq:new S} instead of Eq.~\eqref{eqSSigma}.
\end{theorem}

Let us give two simple examples to illustrate the difference of the computation of $S^{w,\Sigma}_{\lambda,\mu}$ between the two cases discussed in Sections~\ref{regular}
and~\ref{notregular}.

\begin{example}
  We consider first the case of $G=U(4)$ and $K=SU(2)\times SU(2)$ embedded in $G$ by identifying $\C^4=\C^2\otimes \C^2$.
  We identify $i\t_\g$ and $i\t_\g^*$ with $\R^4$, as usual.
  Thus the weights of the representation of $T_G$ on $\C^4$ are $\{ \epsilon_1=[1,0,0,0]$, $\epsilon_2=[0,1,0,0]$, $\epsilon_3=[0,0,1,0]$, $\epsilon_4=[0,0,0,1] \}$, and the roots of $G$ are $\{ \epsilon_i - \epsilon_j, i\neq j\}$.
  Then $\t_\k$ is two-dimensional, and none of the roots of $\g$ vanishes identically on $i\t_\k$.
  Let $\lambda=[k,k,0,0]$  be a highest weight for a representation of $G$.
  If $\Sigma=\{\epsilon_1-\epsilon_2,\epsilon_3-\epsilon_4\}$, then $\lambda\in \Lambda^\Sigma_{G,\geq0}$.

  The restriction of $\Delta_\uu$ to $i\t_\k$ is $[\alpha_2,-\alpha_2,-\alpha_2-\alpha_1,-\alpha_2+\alpha_1]$  where $\alpha_1,\alpha_2$ are the simple roots of $SU(2)\times SU(2).$
  Let $z=[z_1,z_2]$ denote the coordinates of $z\in i\t_\k$ with respect to the dual basis of $\{\alpha_1, \alpha_2\}$.
  For the identity permutation, $w=1$, we have
  $$\frac{e^{\langle \overline{w(\lambda)},z\rangle}}{\prod_{\alpha \in \Delta_{\uu}}(1-e^{-\la \overline{w(\alpha)},z\ra})} = \frac{u_2^k }{\left( 1-\frac{1}{u_2} \right) ^{2} \left( 1-{\frac 1 {u_1u_{{2}}}} \right)  \left( 1-{\frac {u_{{1}}}{u_{{2}}}} \right) },$$
  where we made the change of variables  $e^{z_1}=u_1$, $e^{z_2}=u_2$.
  To obtain the restricted character, we have to sum over 6 permutations $w$, corresponding to $\CW_\g/\CW_\lp \cong \mathfrak S_4/\la(1 2), (3 4)\ra$.
  The result is:
  {\small
  \begin{align*}
    &\qquad \chi_{\lambda}(\exp(z))
    = \frac{u_2^k}{\left( 1-\frac{1}{u_2} \right)^2 \left( 1-\frac{u_1}{u_2}\right)  \left( 1-\frac {1}{u_1u_2}\right) } \\
    &+ \frac{u_1^k}{\left( 1-\frac{1}{u_1}\right)^2 \left( 1-\frac {1}{u_1u_2} \right) \left( 1-\frac{u_2} {u_1} \right) }+\frac {2}{ \left( 1-u_{{2}} \right)  \left( 1-u_{{1}} \right)  \left( 1-\frac{1}{u_{{1}}} \right) \left( 1-\frac{1}{u_2}\right) } \\
    &+ \frac {u_1^{-k}} { \left( 1-u_1 \right) ^{2} \left(1  -u_{{1}}u_{{2}}\right) \left( 1-\frac {u_1}{u_2} \right) }+\frac {u_2^ {-k}}{ \left( 1-u_2\right) ^{2} \left( 1-u_{{1}}u_{{2}} \right) \left( 1-\frac{u_{{2}}}{u_1 }\right) }.
  \end{align*}
  }Moreover, for the permutation $w=1$ and $\mu=[\mu_1,\mu_2]$ (with $\mu_1,\mu_2$ integers), the function $S^{w,\Sigma}_{\lambda,\mu}$ defined by Eq.~\eqref{eqSSigma} is given by
  $$S^{w,\Sigma}_{\lambda,\mu}(z_1,z_2) = \left(1-\frac 1 {u_1}\right) \left(1-\frac1 {u_2}\right) \frac{u_2^k u_1^{-\mu_1}u_2^{-\mu_2}} {\left( 1-\frac{1}{u_2} \right) ^2 \left( 1-\frac {u_1}{u_2}\right)  \left( 1-\frac {1}{u_1u_2}\right) },$$
  and similarly for the other permutations.
\end{example}

\begin{example}
  Consider now $G=U(4)$ and $K_1=SU(2)\times \{1\}$, which is a subgroup of $K=SU(2)\times SU(2)$ discussed in the preceding example.
  Continuing with the above example and notation, we would like to have an expression for the restriction of $\chi_\lambda$ to $T_{K_1}$ as a sum of explicit meromorphic functions.
  For example, we may again consider the identity permutation $w=1$.
  Then the term
  $$\frac{u_2^k }{\left( 1-\frac{1}{u_2} \right) ^{2} \left( 1-{\frac 1 {u_1u_{{2}}}} \right)  \left( 1-{\frac {u_{{1}}}{u_{{2}}}} \right) },$$
  \emph{cannot} be restricted to  $K_1=SU(2)\times \{1\}$, since $u_2-1$ vanishes identically on $K_1$.
  Following the prescription explained above, for $z=[z,0]$ and $Y_1=[0,1]$, $w=id$, we compute that
  \begin{align*}
    &\qquad G_w(z,Y_1) = -\frac{1}{2}\,{\frac { \left( k+4 \right) ^{2}u}{ \left( u-1 \right) ^{2}}}+\frac{1}{2}{\frac { \left( k+4 \right) u}{ \left( u-1 \right) ^{2}}} \\
    &+{\frac { \left( k+4 \right) {u}^{2}}{\left( u-1 \right) ^{3}}}-{\frac { \left( k+4 \right) u}{\left( u-1 \right) ^{3}}}-{\frac {{u}^{3}}{ \left( u-1 \right) ^{4}}}+{\frac {{u}^{2}}{ \left( u-1 \right)^{4}}}-{\frac {u}{ \left( u-1 \right) ^{4}}},
  \end{align*}
  where $u=e^z$.
  And for $\mu$ an integer, our function $S^{w,\Sigma}_{\lambda,\mu}$ in Eq.~\eqref{eq:new S} is given by
  \[ S^{w,\Sigma}_{\lambda,\mu}(z) = \left(1-\frac1u\right) u^{-\mu}G_w(z,Y_1). \]
\end{example}

{\color{red}

}

\begin{remark}\label{heckversussingular}
  Our method of computing $m_{G,K}(\lambda,\mu)$ is a generalization of the Kostant-Heckman branching theorem.
  If $\lambda$ is singular, then the formula obtained for $\chi_\lambda|_{T_K}$ is not very explicit -- but it has two obvious advantages from an algorithmic point of view:
  First, there is a smaller number of elements of the Weyl group over which we sum up.
  Second, the function of which we compute the residues has fewer poles.
  We fully take advantages of these points in our algorithm, in particular when we compute Hilbert series in Section~\ref{subsec:HS}.
\end{remark}

\section{Examples}\label{sec:examples}

Our branching formula (Theorem~\ref{theo:branch}) together with Eq.~\eqref{eq:kron via branching} can be readily turned into an algorithm for computing Kronecker coefficients.
We present this algorithm in Appendix~\ref{app:thealgoKro}.
In this section, we discuss a number of interesting examples of dilated Kronecker coefficients and Hilbert series computed using our algorithm.

\subsection{Dilated Kronecker coefficients}\label{subsec:dilated}

\subsubsection{$\C^4\otimes\C^2\otimes\C^2$}
This example has been studied in complete details by Briand-Orellana-Rosas~\cite{BOR1}.
In particular, a cone decomposition into $74$ cones of quasi-polynomiality is given, together with the corresponding quasi-polynomials, which have degree 2 and period 2.

Our algorithm can reproduce each quasi-polynomial given a point in the interior of its cone of validity (as well as the corresponding dilated Kronecker coefficients).
For example, in the tope defined by $\lambda=[132, 38, 19, 11]$, $\mu=[110, 90]$, $\nu=[120, 80]$, the Kronecker coefficient is given by the quasi-polynomial
\begin{align*}
  g(\lambda,\mu,\nu) &= \frac12\lambda_3\mu_1 + \frac12\lambda_2\lambda_3 - \frac12\nu_1 + \frac12\lambda_2 + \frac12\lambda_3 - \lambda_4 + \frac12\mu_1 \\
  &- \frac14\lambda_3^2 - \frac12\lambda_3\nu_1 - \frac12\lambda_4\mu_1 - \frac12\lambda_2\lambda_4 + \frac18(-1)^{\lambda_2 + \lambda_4 + \mu_1 + \nu_1} \\
  &+ \frac34 + \frac14\lambda_4^2 + \frac12\lambda_4\nu_1 + \frac18(-1)^{\lambda_2 + \lambda_3 + \mu_1 + \nu_1}
\end{align*}
and the dilated Kronecker coefficient for these three Young diagrams is equal to
\[ g(k\lambda,k\mu,k,\nu) = 52 k^2 + \frac{25}2 k + \frac34 + \frac14 (-1)^k. \]
We now present several new examples that have not yet appeared in the literature.

\subsubsection{$\C^2\otimes\C^2\otimes\C^2\otimes \C^2$ ($4$ qubits)}
The Kirwan polytope has been described by Higuchi-Sudbery-Szulc~\cite{H-Su-Sz} (see Example~\ref{ex:N qubit cone}).
We do not know the number of cones in a cone decomposition of $C_K(\CH)$ into cones of quasi-polynomiality.
Nevertheless, given highest weights $\alpha,\beta,\gamma,\delta$, we can compute $g(k\alpha,k\beta,k\gamma,k\delta)$ as a quasi-polynomial in $k$.
It is a quasi-polynomial of degree at most $7$ and period $6$.
More precisely, it is of the form
$$P(k)+(-1)^k Q(k)+R(k),$$
where $P(k)$ is a polynomial of degree at most $7$, $Q(k)$ of degree at most $3$, and $R(k)$ a periodic function of $k$ $\mod$ 3.

Here is an example.
When $\alpha=\beta=\gamma=\delta=[2,1]$, then:
\begin{align*}
&g(k\alpha,k\beta,k\gamma,k\delta)
= \frac{23}{241920} k^7
+ \frac{13}{5760} k^6
+ \frac{155}{6912} k^5
+ \frac{139}{1152} k^4 \\
&\qquad+ \biggl( \frac{81601}{207360} + \frac1{1536} (-1)^k \biggr) k^3
+ \biggl( \frac{9799}{11520} + \frac5{256} (-1)^k \biggr) k^2 \\
&\qquad+ \biggl( \frac{38545}{32256} + \frac{179}{1536} (-1)^k \biggr) k
+ C(k),
\end{align*}
with
\[ C(k) = \biggl(\frac5{243} + \frac1{243}\theta\biggr) (\theta^2)^k + \biggl( \frac4{243} - \frac1{243} \theta \biggr) \theta^k + \frac{51}{256} (-1)^k + \frac{5279}{6912}, \]
where $\theta$ is a third primitive root of unity, $\theta^3=1$.
Thus the degree-zero term $C(k)$ is a sum of a periodic function of period $3$ and of a periodic function of period $2$, leading to periodic behavior modulo $6$.
Its values on $0,1,2,3,4,5$ are:
$$1, \frac{5725}{10368}, \frac{76}{81}, \frac{77}{128}, \frac{77}{81}, \frac{5597}{10368}$$
This determines the Kronecker coefficients $g(k\alpha,k\beta,k\gamma,k\delta)$ for any value of $k$.
Starting from $k=0$, they are
$$1, 3, 13, 39, 110, 264, 588, 1194, 2289, 4134, 7152, 11865, \ldots$$

\subsubsection{$\C^6\otimes \C^3\otimes \C^2$}\label{exa:632}

When $n_2=3$, $n_3=2$, it is sufficient to consider the case when $n_1=6$ (Lemma~\ref{lem:Cauchy reduction}).
In this case, the maximum degree of the quasi-polynomial function $g(k\lambda,k\mu,k\nu)$ is 8 according to the formula in Example~\ref{ex:kron degree}.
We find that the dilated Kronecker coefficient $k\mapsto g(k\lambda,k\mu,k\nu)$ is a quasi-polynomial function of the form
$$P(k)+(-1)^k Q(k)+R(k),$$
where $P(k)$ is a polynomial of $k$ of degree less or equal to $8$, $Q(k)$ of degree less or equal to $2$ and $R(k)$ is a periodic function of $k$ $\mod$ 3.
Here is an example where the degree of the quasi-polynomial is the maximum one:

We fix $\lambda=[15,10,9,4,3,2]$, $\mu=[21,14,8]$, $\nu=[27,16]$ and compute:
\begin{align*}
  &g(k\lambda,k\mu,k\nu)
= {\frac {413587}{967680}} k^8
+{\frac {66773}{17280}} k^7
+{\frac {3072191}{207360}} k^6
+{\frac {1091771}{34560}} k^5 \\
&\qquad+{\frac {710713}{17280}} k^4
+{\frac {871363}{25920}} k^3
+\left( (-1)^k\frac{55}{1024}+\frac {1833073}{107520}) \right)  k^2 \\
&\qquad+\left((-1)^k\frac{79}{512}+\frac {117661}{23040}\right) k
+\left( {\frac {10}{243}}+{\frac {4}{243}}\,\theta \right) \theta^k \\
&\qquad+ \left( {\frac {2}{81}}-{\frac {4}{243}}\,\theta \right) (\theta^2)^k
+{\frac {275}{2048}} (-1)^k
+\frac {398071}{497664}
\end{align*}
where $\theta$ is a third primitive root of unity.
The degree-zero term is thus a periodic function $C(k)$, of period $6$, whose values on $0,1,2,3,4,5$ are given by
$$1,{\frac {50429}{82944}},{\frac {25}{27}},{\frac {749}{1024}},{\frac{71}{81}},{\frac {18175}{27648}}.$$
Thus, the values of $g(k\lambda,k\mu,k\nu)$, starting from $k=0$, are
$$1,148,3570,34140,197331,829417,2797696,\ldots$$

\medskip

We now illustrate some other particularly interesting examples that connect the behavior of the quasi-polynomial function $g(\lambda,\mu,\nu)$ on cones $\c_a$ adjacent to a facet $F$ of the Kirwan cone (cf.\ Section~\ref{subsec:faces}).

Recall that a wall of the Kirwan polytope is called \emph{regular} if it intersects the interior of the Weyl chamber.
Since the Kirwan cone for the action of $U(6)\times U(3)\times U(2)$ in $\C^{6}\otimes \C^3 \otimes \C^2$ is solid, it is determined by the inequalities defined by regular walls, together with those determining the Weyl chamber.
The regular walls have been described by Klyachko~\cite{Klyachko}.
There are 5 types of inequalities in $\lambda=[\lambda_1, \lambda_2,\lambda_3,\lambda_4,\lambda_5,\lambda_6]$, $\mu=[\mu_1,\mu_2,\mu_3]$, $\nu=[\nu_1,\nu_2]$, which we recall in Table~\ref{Wt} (as well as the equations $\lvert\lambda\rvert=\lvert\mu\rvert=\lvert\nu\rvert$).

More precisely, for each of the inequalities in Table~\ref{Wt}, there is a particular subset $S$ of $\mathfrak {S}_6\times \mathfrak {S}_3 \times \mathfrak {S}_2$ computed by Klyachko such that the permuted inequality is an irredundant inequality of the Kirwan cone.
Each subset $S$ contains the identity, so the inequalities in Table~\ref{Wt} are particular examples of regular walls of the Kirwan cone, which we denote by $F_I$, $F_{II}$, etc.

\begin{table}
{\tiny
\begin{center}
\begin{tabular}{|c|l|}\hline
Type& Inequality \\
\hline
\hline
I   & $\makebox[0pt][l]{$F_{I}$}\phantom{F_{III}}: \nu_1-\nu_2-\lambda_1-\lambda_2-\lambda_3+\lambda_4+\lambda_5+\lambda_6 \leq 0$ \\
\hline
II  & $\makebox[0pt][l]{$F_{II}$}\phantom{F_{III}}: \mu_1+\mu_2-2\,\mu_3-\lambda_1-\lambda_2-\lambda_3-\lambda_4+2\,\lambda_5+2\,\lambda_6\leq 0$ \\
\hline
III & $\makebox[0pt][l]{$F_{III}$}\phantom{F_{III}}: 2\,\mu_1-2\,\mu_3+\nu_1-\nu_2-3\,\lambda_1-\lambda_2-\lambda_3+\lambda_4+\lambda_5+3\,\lambda_6\leq 0$ \\
\hline
IV  & $\makebox[0pt][l]{$F_{IV}$}\phantom{F_{III}}: 2\,\mu_1+2\,\mu_2-4\,\mu_3+3\,\nu_1-3\,\nu_2-5\,\lambda_1-5\,\lambda_2+\lambda_3+\lambda_4+\lambda_5+7\,\lambda_6\leq 0$ \\
\hline
V   & $\makebox[0pt][l]{$F_{V}$}\phantom{F_{III}}: 4\,\mu_1-2\,\mu_2-2\,\mu_3+3\,\nu_1-3\,\nu_2-7\,\lambda_1-\lambda_2-\lambda_3-\lambda_4+5\,\lambda_5+5\,\lambda_6\leq 0$ \\
\hline
\end{tabular}
\end{center}}
\caption{Types of regular walls for $\C^6 \otimes \C^3 \otimes \C^2$~\cite{Klyachko}.}\label{Wt}
\end{table}

In Table~\ref{weq}, we give a point $v_F=[\lambda_F,\mu_F,\nu_F]$ in the relative interior of the facet $F \cap C_K(\CH)$ for each $F\in \{F_I,F_{II}, F_{III}, F_{IV},F_V\}$.
We also display the corresponding dilated Kronecker coefficients $g(k v_F)$, which is of the maximum degree (as predicted by \cite[Lemma 37]{BV2015}) among the dilated coefficients $g(k\lambda, k\mu,k\nu)$ when $(\lambda, \mu,\nu)$ varies in a facet of the given type $F$.

\begin{table}
{\tiny
\begin{center}
\begin{tabular}{|l|l|l|}
\hline
Facet & $v_F=[\lambda_F,\mu_F,\nu_F]$ & $g(k\lambda_F,k\mu_F,k\nu_F)$ \\
\hline\hline
$F_I$ & $[[288, 192, 174, 120, 30, 6], [343, 270, 197], [654, 156]]$ & $1+17k$ \\
\hline
$F_{II}$ & $[[300, 186, 150, 78, 48, 6], [438, 276, 54], [465, 303]]$ & \raisebox{-.05cm}{$\frac{121077}{4} k^{3} + \frac{21051}{8} k^2$} \\[.1cm]
&& \raisebox{-.05cm}{$+ \frac{311}{4} k + \frac 3 {16} (-1)^k + \frac{13}{16}$} \\[.1cm]
\hline
$F_{III}$ & $[[47, 35, 23, 13, 5, 1], [76, 38, 10], [85, 39]]$ & $1$ \\
\hline
$F_{IV}$ & $[[276, 204, 120, 66, 30, 6], [351, 273, 78], [552, 150]]$ & $1+36k$ \\
\hline
$F_V$ & $[[276, 198, 126, 66, 48, 6], [406, 201, 113], [536, 184]]$ & $1+41k$ \\
\hline
\end{tabular}
\end{center}}
\caption{Dilated Kronecker coefficients for points in the interior of regular walls.}\label{weq}
\end{table}

Let $C_{III}=F_{III} \cap C_K(\CH)$.
From the reduction principle of multiplicities on regular faces, we know that the restriction of $g(\lambda,\mu,\nu)$ to $C_{III}\cap \Lambda_K$ (or to any facet obtained by the Klyacho permutations) is identically $1$.
Indeed, let $X_0=[[-3,-1,-1,1,1,3],[2,0,-2],[1,-1]]$.
This is an element of $i\t_\k$ perpendicular to the wall associated to $C_{III}$.
Let $K_0$ denote the stabilizer of $X_0$ in $K$, and $\CH^0$ the subspace of $\CH$ stable by $X_0$.
Then $K_0$ is isomorphic to the subgroup
$$\bigl( U(1)\times U(2)\times U(2)\times U(1) \bigr) \times \bigl( U(1)\times U(1)\times U(1) \bigr)  \times \bigl( U(1)\times U(1) \bigr)$$
of $U(6)\times U(3)\times U(2)$.
The multiplicity function $m_K^{\CH}$ restricted to $C_{III}\cap \Lambda_K$ coincides with $m_{K_0}^{\CH_0}$.
It is easily computed to be identically $1$ on $C_{III}\cap \Lambda_K$.
Thus, any element of $C_{III}\cap \Lambda_K$ (or any facet obtained by the Klyacho  permutations) is \emph{stable}.
We recall that a point $\lambda$ is called \emph{stable} if the function $k\mapsto m_K^\CH(k\lambda)$ is a bounded function of $k$ (then it necessarily takes the value zero or one~\cite{pep-RR}).

\medskip

For each of the cases in Table \ref{weq} we can also compute symbolically a quasi-polynomial function coinciding with the Kronecker coefficients on a closed solid cone $\c_{v_F}$ of $C_K(\CH)$ containing the element $v_F$.
Following the general method explained in Appendix~\ref{app:thealgoKro}, we compute an element $v_F^{\epsilon}$ close to $v_F$ that is not on any admissible wall.
For example for $F_I$ and $v_{F_I}$, we can choose
\[ v_{F_I}^{\epsilon}=[[291, 194, 175, 120, 30, 6], [347, 272, 197], [659, 157]]. \]
Then the function $g(\lambda,\mu,\nu)$ is on the tope $\tau(v_F^{\epsilon})$ containing $v_F^{\epsilon}$ given by a quasi-polynomial function.
The closure $\c_{v_F}$ of $\tau(v_F^{\epsilon})$ contains $v_F$ and $\c_{v_F}$ is a cone of quasi-polynomiality adjacent to the facet $F \cap C_K(\CH)$.
The degree of the quasi polynomial function $g(\lambda,\mu,\nu)$ on $\c_{v_F}$ is $8$, as we already know.
When we restrict this quasi-polynomial to the ``ray'' $k v_F$, we do indeed get $g(k v_F)$.

We summarize our results in Table~\ref{wall1}.
We find that, remarkably, the symbolic function on $\c_{v_{F_I}}$ is polynomial, instead of merely quasi-polynomial (first row).
It is a striking fact that this polynomial function is divisible by $7$ linear factors with constant values $1,2,3,4,5,6,7$ on the face $F_I$.
Thus the restriction of $g(\lambda,\mu,\nu)$ to $F_I$ is a linear polynomial (third row).
When evaluated on $v_{F_I}$ it gives indeed $1+17k$, as we had previously computed independently directly using the element $v_{F_I}.$

\begin{table}
{\tiny
\begin{center}
\begin{tabular}{|l|l|}
\hline
$[\lambda,\mu,\nu] \in \c_{v_{F_I}}$&$g(\lambda,\mu,\nu)  $\\
\hline
\hline
& $\frac{1}{5040} \left( \lambda_1+\lambda_2+\lambda_3-\nu_1+7\right)\left( \lambda_1+\lambda_2+\lambda_3-\nu_1+6\right)$ \\ 
$[\lambda,\mu,\nu]$ & $\left( \lambda_1+\lambda_2+\lambda_3-\nu_1+5\right)\left( \lambda_1+\lambda_2+\lambda_3-\nu_1+4\right)$ \\
& $\left( \lambda_1+\lambda_2+\lambda_3-\nu_1+3\right)\left( \lambda_1+\lambda_2+\lambda_3-\nu_1+2\right)$ \\
& $\left( \lambda_1+\lambda_2+\lambda_3-\nu_1+1\right)\left( \lambda_1+\lambda_2+\lambda_4+\lambda_5-\mu_1-\mu_2+1\right)$ \\
\hline
$k\ v_{F_I}^\epsilon$ & $ \frac{1}{5040}(k+7)(k+6)(k+5)(k+4)(k+3)(k+2)(k+1)(16k+1)$ \\
\hline
$[\lambda,\mu,\nu] \in F_I \cap \c_{v_{F_I}}$ & $1+\frac{503}{140}(\lambda_1+\lambda_2)+\lambda_4+\lambda_5+\frac{363}{140}(\lambda_3+\nu_1)$ \\
\hline
$k\ v_{F_I}$ & $1+17k$ \\
\hline
\end{tabular}
\end{center}}
\caption{Results for the regular wall of type $I$}\label{wall1}
\end{table}

We do not obtain such nice expressions in the other cases (in particular, the quasi polynomials obtained are not polynomials), but we can nonetheless compute the symbolic quasi-polynomials.
The results of the computations are too long to be included here.

\bigskip

Let us remark that we do not know how to compute the degree when $(\lambda,\mu,\nu)$ is on a face defined by the Weyl chamber (e.g., see Section~\ref{subsec:HS} for the case of three rectangular tableaux).
Here is an example for which the degree is smaller than expected for singular $\mu$.
Consider $\lambda=[9,7,5,3,2,1]$, $\mu=[9,9,9]$, $\nu=[14,13]$.
Then the dilated Kronecker coefficient is given by the following formula:
\begin{align*}
  g(k\lambda,k\mu,k\nu)
= \frac {55}{288} k^5
+ \frac {19}{24} k^4
+ \frac {617}{432} k^3
+ \frac {17}{12} k^2
+ \left( {\frac {13}{64}} (-1)^k + \frac {67}{64} \right) k \\
+ \frac {1}{81} \theta^k \left( -2\,\theta+8 \right)
+ \frac {1}{81} (\theta^2)^k \left( 2\,\theta+10 \right)
+ \frac {85}{144}
+ \frac{3}{16} (-1)^k
\end{align*}
Here $\theta$ is  again a third primitive root of unity, $\theta^3=1$.
Thus the term of degree zero is a periodic function $C(k)$ that takes the following values on $0,1,2,3,4,5$:
$$1, 71/216, 17/27, 5/8, 19/27, 55/216$$
Of course, the value of $g(0,0,0)$ is equal to $1$.
Here, $g(\lambda,\mu,\nu)=5$,
and, for instance, $g(17\lambda, 17\mu, 17\nu)=344715.$

\subsubsection{$\C^3\otimes\C^3\otimes\C^3$ (3 qutrits)}\label{C333}
In this case, the multiplicity function $k\mapsto g(k\lambda,k\mu,k\nu)$ is a quasi-polynomial function of degree at most $11$, whose constant term is a periodic function of $k$ $\mod$ 12.
The actual numerical values can be computed rather quickly using the algorithm in Appendix~\ref{app:thealgoKro}.

Let us give an example of the dilated Kronecker coefficient for $\lambda=\mu=\nu=[4,3,2]$.
Here, $g(k\lambda,k\mu,k\nu)$ has precisely degree 11.
We omit the full formula as it is too long.
The periodic term for the coefficient of degree zero is given on $k=0,\dots,11$ by the values:
{\footnotesize{
$$1, \frac{1166651}{5308416}, \frac{13403}{20736}, \frac{29899}{65536}, \frac{59}{81}, \frac{1166651}{5308416}, \frac{235}{256}, \frac{980027}{5308416}, \frac{59}{81}, \frac{32203}{65536}, \frac{13403}{20736}, \frac{980027}{5308416}$$
}}

\subsection{Rectangular tableaux and Hilbert series}\label{subsec:HS}

In Table~\ref{tab:hilbert} we give a list of the Hilbert series associated with the Kronecker coefficients for rectangular partitions (cf.\ Example~\ref{Hilbertseries}).
We use the following notation:
In the first column, $(\C^2)^3=\C^2\otimes\C^2\otimes\C^2$, $[[1,1]]^3=[[1,1],[1,1],[1,1]],$ $\C^{[4,3,3]}=\C^4\otimes \C^3\otimes \C^3$, and similarly.
The second column refers to the choice of the parameters $[\lambda,\mu,\nu]$, and the third column gives teh Hilbert series $\sum_k m(k)t^k$, where $m(k)=g(k\lambda,k\mu,k\nu)$ is the Kronecker coefficient.

The Hilbert series in the third row for $(\C^2)^5$, the case of $5$ qubits, is
\begin{equation}\label{eq:HS22222}
  HS_{22222} = \frac {P(t)} {(1-t^2)^5(1-t^3)(1-t^4)^5(1-t^{5})(1-t^{6})^5},
\end{equation}
where
{\small
\begin{align*}
&P(t)=t^{52}+16\, t^{48}+9\, t^{47}+82\, t^{46}+145\, t^{45}+383\, t^{44} + 770\, t^{43} \\
&+1659\, t^{42} +3024\, t^{41}+5604\, t^{40}+9664\, t^{39} +15594\, t^{38}+24659\, t^{37} \\
&+36611\, t^{36} +52409\, t^{35}+71847\, t^{34}+95014\, t^{33}+119947\, t^{32} \\
&+146849\, t^{31} +172742\, t^{30} +195358\, t^{29}+214238\, t^{28}+225699\, t^{27} \\
&+229752\, t^{26}+225699\, t^{25} +214238\, t^{24} +195358\, t^{23} +172742\, t^{22} \\
&+146849\, t^{21}+119947\, t^{20} +95014\, t^{19} +71847\, t^{18} +52409\, t^{17} \\
&+36611\, t^{16}+24659\, t^{15}+15594\, t^{14} +9664\, t^{13}+5604\, t^{12} +3024\, t^{11} \\
&+1659\, t^{10} +770\, t^{9} +383\, t^{8}+145\, t^{7}+82\, t^{6}+9\, t^{5}+16\, t^4+1.
\end{align*}
}
We remark that the result in~\cite{LUTHI} corresponds to the series $\sum_k m(k) t^{2k}$ and has a misprint on the value of the coefficient $a_n$ for $n=42$ (corresponding to the coefficient of $t^{21}$ in our formula for $P$), as the numerator given is not palindromic (cf.\ Remark~\ref{rem:gorenstein}).
That is, the coefficient $a_{42}$ in~\cite{LUTHI} has to be replaced by $146849.$

\begin{table}
\begin{center}
\begin{tabular}{|l|c|l|}
\hline
Type & Parameters & Hilbert series \\
\hline
\hline
$(\C^2)^3$&$[[1,1]]^3 $ &$\frac{1}{1-t^2}$ \\
\hline
$(\C^2)^4$&$[[1,1]]^4 $ &$\frac{1}{(1-t)(1-t^2)^2(1-t^3)}$ \\
\hline
$(\C^2)^5$&$[[1,1]]^5$ &$HS_{22222}$ {\footnotesize (see Eq.~\eqref{eq:HS22222})} \\
\hline
$(\C^3)^3$&$[[1,1,1]]^3 $ &$\frac{1}{(1-t^2)(1-t^3)(1-t^4)}$ \\
\hline
$\C^{[4,3,3]}$&$[[3,3,3,3], [4,4,4], [4,4,4]]$&$\frac {1+{t}^{9}}{ \left( 1-{t}^{2} \right) ^{2} \left(1- {t}^{4} \right)  \left( 1-t\right)  \left(1- {t}^{3} \right) }$ \\
\hline
\end{tabular}
\end{center}
\caption{Hilbert series associated with rectangular Kronecker coefficients}\label{tab:hilbert}
\end{table}

\medskip

We now give the dilated Kronecker coefficients in the first, second and fourth example considered in Table~\ref{tab:hilbert} (we omit the other two because the formulae are too long to be reproduced here).
The results are as follows:
{\small
\begin{align*}
&g(k[1,1],k[1,1]) = \frac12 + \frac12 (-1)^k, \\
&g(k[1,1],k[1,1],k[1,1], k[1,1])
= \frac {23}{36}
+ \frac{1}{4} (-1)^k
+ \frac{1}{27} \theta^k \left( 2+\theta \right) \\
&\qquad+ \frac{1}{27} (\theta^2)^k \left(1-\theta \right)
+ \left( \frac {29}{48} + \frac{1}{16} (-1)^k \right) k
+\frac{1}{16} k^2
+\frac {{k}^{3}}{72}, \\
&g(k[1,1,1],k[1,1,1],k[1,1,1])
= \frac {107}{288}
+ \frac {9}{32} (-1)^k \\
&\qquad+ \left(1 + (-1)^k \right) \frac{1}{16}{i}^{k}
+ \left( 1+ (-1)^{k+1} \right) \frac{1}{16} i^{k+1}
+ \frac{1}{9} (\theta^2)^k \\
&\qquad+ \frac{1}{9} \theta^k
+ \left( \frac{1}{16} (-1)^k + \frac{3}{16} \right) k
+ \frac{1}{48} k^2,
\end{align*}
}
where $\theta$ is a third primitive root of unity.
As an example, we show how the latter coefficient, $g(k[1,1,1],k[1,1,1],k[1,1,1])$, can equivalently be expressed by a list of polynomials on the cosets.
We have 12 cosets, and thus a sequence of 12 polynomials for $k=0,\dots,12$, given by the following list:
{\small{
\begin{align*}
&1+\frac{1}{4}k+\frac{1}{48}{k}^{2}, \;\;
-{\frac {7}{48}}+\frac{1}{8}k+\frac{1}{48}{k}^{2}, \;\;
{\frac {5}{12}}+\frac{1}{4}k+\frac{1}{48}{k}^{2}, \;\;
{\frac {7}{16}}+\frac{1}{8}k+\frac{1}{48}{k}^{2}, \\
&\frac{2}{3}+\frac{1}{4}k+\frac{1}{48}{k}^{2}, \;\;
-{\frac {7}{48}}+\frac{1}{8}k+\frac{1}{48}{k}^{2}, \;\;
\frac{3}{4}+\frac{1}{4}k+\frac{1}{48}{k}^{2}, \;\;
{\frac {5}{48}}+\frac{1}{8}k+\frac{1}{48}{k}^{2}, \\
&\frac{2}{3}+\frac{1}{4}k+\frac{1}{48}{k}^{2}, \;\;
\frac{3}{16}+\frac{1}{8}k+\frac{1}{48}{k}^{2}, \;\;
{\frac {5}{12}}+\frac{1}{4}k+\frac{1}{48}{k}^{2}, \;\;
{\frac {5}{48}}+\frac{1}{8}k+\frac{1}{48}{k}^{2}
\end{align*}
}}
The following are the values of the Kronecker coefficients computed by the above formula for $k=0,\dots,20$:
$$1, 0, 1, 1, 2, 1, 3, 2, 4, 3, 5, 4, 7, 5, 8, 7, 10, 8, 12, 10, 14$$
They are part of what is known as sequence A005044 in the \emph{on-line encyclopedia of integers sequences (OEIS)}.

Observe that in this example the saturation factor is 2.
We recall that the \emph{saturation factor} of a given $\lambda\in C_K(\CH)\cap\Lambda_{K,\geq0}$ is the smallest positive $k$ such that $m_K^\CH(k\lambda)>0$.

\subsubsection{The Hilbert series of measures of entanglement for $4$ qubits}\label{Wallach.ex}
Consider $\CH=(\C^2)^4=\C^2\otimes\C^2\otimes\C^2\otimes \C^2$ and consider the standard action of $U(2)\times U(2)\times U(2)\times U(2)$ on $\CH$.
The space $\CH$ is the space of $4$ qubits.

We now consider the direct sum $\tilde \CH=\CH\oplus \CH$ of two copies of $\CH$, so $\Sym(\tilde\CH)=\Sym(\CH)\otimes \Sym (\CH)$.
The decomposition of  the tensor product representation of $U(2)\times U(2)\times U(2)\times U(2)$ in $\Sym(\tilde \CH)=\Sym(\CH)\otimes \Sym (\CH)$ has been considered by Wallach~\cite{Nolan}.
The Hilbert series of the invariants is called the \emph{Hilbert series of measures of entanglement for $4$ qubits}.

To compute this Hilbert series, it is useful to think of
\[ \tilde \CH=\CH\oplus \CH=\CH\otimes \C^2 = (\C^2\otimes\C^2\otimes\C^2\otimes \C^2) \otimes \C^2 \]
with the action of $K=U(2)\times U(2)\times U(2)\times U(2)\times \{1\}$.
We first regroup
\[ \tilde \CH=\C^2\otimes(\C^2\otimes\C^2\otimes \C^2\otimes \C^2) \]
and consider the action of $U(2)\times U(16)$.
As in the case of $5$ qubits, the Cauchy formula allows us to compute $\Sym(\tilde \CH)$ as a representation of $U(2)\times U(16)$.
Consider  $\lambda=[1,1]$ and $\tilde{\lambda}=[1,1,0,0,\ldots, 0]$, a highest weight for $U(16)$.
Let $K=U(2)\times U(2)\times U(2)\times U(2)$ embedded in $G=U(16)$ and let $K_1=SU(2)\times SU(2)\times SU(2) \times \{1\}$ embedded in $K$.
Thus following the method outlined in Section~\ref{notregular} and Theorem~\ref{theo:branchingsingular}, we can compute the branching coefficient $m(k)=m_{G,K_1}(k\tilde{\lambda},k\mu)$, with $\mu=0$ indexing the trivial representation of $K_1$.
We obtain a quasi-polynomial $m(k)$ of degree 19 and period 6, that we list at the end of this section. 
Its generating function is the Hilbert series of measures of entanglement for $4$ qubits as computed by Wallach in~\cite{Nolan}.
We have recomputed his formula:
\begin{align*}
  \sum_k m(k) t^k = \frac{P(q)}{(1-q^2)^3(1-q^4)^{11}(1-q^6)^6}
\end{align*}
where $m(k) = \dim \left[\Sym^{2 k}(\tilde \CH)\right]^{SL(\C^2)\times SL(\C^2)\times SL(\C^2)\times SL(\C^2)}$, $q=t^2$, and
{\small
\begin{align*}
&P(q)
=q^{54}+3 q^{50}+20 q^{48}+76 q^{46}+219 q^{44}+654 q^{42}+1539 q^{40} \\
&\qquad+3119 q^{38}+5660 q^{36}+9157 q^{34}+12876 q^{32}+16177 q^{30} \\
&\qquad+18275 q^{28}+18275 q^{26}+16177 q^{24}+12876 q^{22}+9157 q^{20} \\
&\qquad+5660 q^{18}+3119 q^{16}+1539q^{14}+654 q^{12}+219 q^{10}+76 q^{8} \\
&\qquad+20 q^{6}+3 q^{4}+1.
\end{align*}
}

\medskip

We conclude this section by describing the quasi-polynomial $m(k)$.
First, define:
{\tiny
\begin{align*}
&p(k)={\frac {353}{472956150389538816000}} k^{19}+{\frac {353}{3111553620983808000}}k^{18}+{\frac {271067}{33189905290493952000}}k^{17} \\
&\quad+{\frac {90331}{244043421253632000}} k^{16}+{\frac {96329}{8134780708454400}}k^{15}+{\frac {4335209}{15252713828352000}}k^{14} \\
&\quad+{\frac{299075479}{56317712596992000}}k^{13}+\frac{556811179}{7039714074624000}k^{12}+{\frac {1343229996}{14079428149248000}}\,{k}^{11} \\
&\quad+{\frac {1507096313}{159993501696000}}\,{k}^{10}, \\
&p_{\text{even}}(k)={\frac {30016136009\,{k}^{9}}{391095226368000}}+{\frac {8474560763\,{k}^{8}}{16295634432000}}+{\frac {417926105131\,{k}^{7}}{141228831744000}}+{\frac {84164633999\,{k}^{6}}{5884534656000}}, \\
&p_{\text{odd}}(k)={\frac {1920961135001\,{k}^{9}}{25030094487552000}}+{\frac {542157180107\,{k}^{8}}{1042920603648000}}+{\frac {6671912967271\,{k}^{7}}{2259661307904000}}+{\frac {1335013209659\,{k}^{6}}{94152554496000}}
\end{align*}
}

Then the quasi-polynomial $m(k)$ is on the six cosets given by the following 6 polynomials $W_0,\dots,W_5$:
{\tiny
\begin{align*}
&W_0(k)=p(k)+p_{\text{even}}(k)+{\frac {38627139511}{653837184000}}\,{k}^{5}+{\frac {50415619753}{245188944000}}\,{k}^{4}+{\frac {266225257897}{463134672000}}\,{k}^{3} \\
&\qquad+{\frac {4572054901}{3859455600}}\,{k}^{2}+{\frac {14055407\,k}{8953560}}+1 \\
&W_1(k)=p(k)+p_{\text{odd}}(k)+ {\frac {219573425545427}{3813178457088000}}\,{k}^{5}+{\frac {276452038823221}{1429941921408000}}\,{k}^{4}+{\frac {12577822401820393489}{24892428967870464000}}\,{k}^{3} \\
&\qquad+{\frac {572824001947094231}{622310724196761600}}\,{k}^{2}+{\frac {159318923928183241}{166314250686431232}}\,k+{\frac {290588607887}{835884417024}} \\
&W_2(k)=p(k)+p_{\text{even}}(k)+{\frac {38627139511\,{k}^{5}}{653837184000}}+{\frac {36750520335937\,{k}^{4}}{178742740176000}}+{\frac {1745362160646217\,{k}^{3}}{3038626582992000}} \\
&\qquad+{\frac {89590754414783\,{k}^{2}}{75965664574800}}+{\frac {815186343623}{528698764440}}k+{\frac {1506571}{1594323}} \\
&W_3(k)=p(k)+p_{\text{odd}}(k)+{\frac {301217799563}{5230697472000}}\,{k}^{5}+{\frac {379529711549}{1961511552000}}\,{k}^{4}+{\frac {1927034414248049}{3793999233024000}}\,{k}^{3} \\
&\qquad+{\frac {29795123615357}{31616660275200}}\,{k}^{2}+{\frac {109432200819\,k}{104316534784}}+{\frac {261589}{524288}} \\
&W_4(k)=p(k)+p_{\text{even}}(k)+{\frac {28157390911519}{476647307136000}}\,{k}^{5}+{\frac {36724846687937}{178742740176000}\,{k}^{4}}+{\frac {1738714367494217}{3038626582992000}}\,{k}^{3} \\
&\qquad+{\frac {88327521243583}{75965664574800}}\,{k}^{2}+{\frac {2345378642869}{1586096293320}}k+{\frac {1353103}{1594323}} \\
&W_5(k)=p(k)+p_{\text{odd}}(k)+{\frac {301217799563}{5230697472000}}\,{k}^{5}+{\frac {276657428007221}{1429941921408000}}\,{k}^{4}+{\frac {12632281123321577489}{24892428967870464000}}\,{k}^{3} \\
&\qquad+{\frac {583172408085564631}{622310724196761600}}\,{k}^{2}+{\frac {56607866326977347\,k}{55438083562143744}}+{\frac {371050038671}{835884417024}}
\end{align*}
}

The complete quasi-polynomial, with $\theta$ a third primitive root of unity, reads as follows:
{\tiny
\begin{align*}
&m(k) = {\frac {353}{472956150389538816000}}{k}^{19}+{\frac {353}{3111553620983808000}}{k}^{18}+{\frac {271067}{33189905290493952000}}{k}^{17} \\
&\qquad+{\frac {90331}{244043421253632000}}{k}^{16}+{\frac {96329}{8134780708454400}}{k}^{15}+{\frac {4335209}{15252713828352000}}{k}^{14} \\
&\qquad+{\frac {299075479}{56317712596992000}}{k}^{13}+{\frac {556811179}{7039714074624000}}{k}^{12}+{\frac {13432299961}{14079428149248000}}{k}^{11} \\
&\qquad+{\frac {1507096313}{159993501696000}}{k}^{10}+\left(\left( -1 \right) ^{k} {\frac {17}{11890851840}}+{\frac {3841993839577}{50060188975104000}} \right) {k}^{9} \\
&\qquad+ \left(\left( -1 \right) ^{k} {\frac {17}{165150720}}+{\frac {1084529068939}{2085841207296000}} \right) {k}^{8} \\
&\qquad+ \left( \left( -1\right) ^{k} {\frac {817}{247726080}}+{\frac {13358730649367}{4519322615808000}}\right) {k}^{7} \\
&\qquad+ \left(\left( -1 \right) ^{k} {\frac { 91}{1474560}}+{\frac {2681647353643}{188305108992000}} \right) {k}^{6} \\
&\qquad+\left({\frac {2\,(\theta+1)}{1594323}}{\theta}^{k}-{\frac {2\,\theta}{1594323}}\, ( {\theta}^{k} ) ^{2}+{\left( -1 \right) ^{k}\frac {1649}{2211840}}+{\frac {1334555059856737}{22879070742528000}}\right){k}^{5} \\
&\qquad+\left({\frac {(229\,\theta +251)}{4782969}}{\theta}^k+{\frac {(-229\,  \theta+22) }{4782969}}( {\theta}^{k} ) ^{2}+\left( -1 \right) ^{k}{\frac {559}{92160}}+{\frac {570537819046717}{2859883842816000}}\right)k^4 \\
&\qquad+\left({\frac { \left( 3488\,\theta+4192 \right) }{4782969}}{\theta}^{k}+{\frac { ( -3488\,\theta+704) }{4782969}}({\theta}^{k} ) ^{2} +\left( -1 \right) ^{k}{\frac {198924917}{5945425920}}+{\frac {8967103302680393051}{16594952645246976000}}\right )k^3 \\
&\qquad+\left({\frac { \left( 26512\,\theta+34928 \right) }{4782969}}{\theta}^{k}+{\frac {  \left( -26512\,\theta+8416\right) }{4782969}}( {\theta}^{k} ) ^{2}+ \left( -1 \right) ^{k}{\frac {10001959}{82575360}}+{\frac {437463645838719389}{414873816131174400}}\right)k^2 \\
&\qquad+\left({\frac {\left( 100700\,\theta+145244 \right) }{4782969}}{\theta}^{k} +
{\frac { \left( -100700\,\theta+14848\right) }{4782969}}( {\theta}^{k} ) ^{2} \left( -1 \right) ^{k}{\frac {688047337}{2642411520}}+{\frac {6335305750969416391}{4989427520592936960}}\right)k \\
&\qquad+\left({\frac { \left( 5684\,\theta+241220 \right) }{177147}}{\theta}^{k}+{\frac { 1}{177147}}( {\theta}^{k} ) ^{2}+\left( -1 \right) ^{k}{\frac {262699}{1048576}}+{\frac {3413873184941}{5015306502144}}
\right)
\end{align*}
}

\textbf{Acknowledgments.}
We are grateful to Michel Duflo and Robert Zeier for suggestions and comments.
Part of the work for this article was made during the period the authors spent at  the Institute for Mathematical Sciences (IMS) of the National University of Singapore in November/December 2013.
The support received is gratefully acknowledged.
The first author was also partially supported by a PRIN2012 grant.
We are thankful to Michael Walter and the referees who made many comments on the first version of this text.
These comments led (we hope) to an improved exposition of our results.

\begin{appendix}

\section{The algorithm for Kronecker coefficients}\label{app:thealgoKro}

We now give our algorithm for computing Kronecker coefficients.
Below, we recall the notation used, and we state the correctness of our algorithm in Theorem~\ref{theo:correct}.
The assumptions of our algorithm are natural and without loss of generality, as explained in Remark~\ref{rem:algo assumptions}.
We discuss some useful variants in Remark~\ref{rem:algo variants} and conclude by reporting on our experiences with our \textsc{Maple} implementation, available at~\cite{Kroneckerwebpage}.

\medskip

\begin{algorithmic}[1]
\Require Young diagrams $\nu^0_j \in P\Lambda_{U(n_j),\geq 0}$ for $j=1,\dots,s$, each specified by a list of integers $[\nu^0_{j,1},\dots,\nu^0_{j,n_j}]$ with $\nu^0_{j,1}\geq\dots\geq\nu^0_{j,n_j}\geq0$.
\Statex \hspace{-.7cm}\textbf{Assumptions:} $s\geq3$, \; $\lvert\nu^0_1\rvert=\dots=\lvert\nu^0_s\rvert$, \; $n_1\geq n_2\geq\dots\geq n_s\geq 2$, \; $\nu^0_{1,2}>0$, \; $n_1\leq M=n_2\dots n_s$.
\Ensure A quasi-polynomial function $g(\nu_1,\dots,\nu_s)$ that agrees with the Kronecker coefficient on a closed polyhedral cone that contains $\nu^0_1,\dots,\nu^0_s$.
\Statex
\State\label{algoline:lambdamu0} $\lambda^0 \gets {\tilde\nu}^0_1 = (\nu^0_1,0,\dots,0)$, \; $\mu^0 \gets (\nu^0_2,\dots,\nu^0_s)$
\State\label{algoline:lambdamu} $\lambda \gets {\tilde\nu}_1 = (\nu_1,0,\dots,0)$, \; $\mu \gets (\nu_2,\dots,\nu_s)$
\Statex
\State\label{algoline:N} $\CN \gets$ normal vectors $X$, with integer coefficients, of the admissible hyperplanes $H\in\CA(\Psi)$
\State\label{algoline:epsdelta}$(\varepsilon,\delta) \gets $ a point in the interior of $C_{G,K}^\Sigma \setminus \bigcup_{F \in \CF_\Sigma} F$, rescaled such that $\lvert \langle \overline{w(\varepsilon)},  X \rangle - \langle \delta, X \rangle \rvert < \frac12$ for all $X\in\CN$, $w\in S_M$
\State $(\lambda^1, \mu^1) \gets (\lambda^0, \mu^0) + (\varepsilon, \delta)$
\Statex
\State\label{algoline:g0} $g \gets 0$
\State\label{algoline:q} $q \gets$ index of $\Psi\setminus\Delta_\k^+$ in $\Lambda_K$
\State\label{algoline:Y} $Y \gets (Y_2,\dots,Y_s)$, where $Y_j = \left((n_j-1),(n_j-3),\dots,1-n_j\right) \prod_{i=2}^j n_j$
\ForAll{$[w] \in S_M / S_{\{n_1+1,\dots,M\}}$}
  \State\label{algoline:polari} $\Psi_{w,\uu} \gets \{\hphantom{-}\overline{w(\alpha)}, \ \alpha \in \Delta_\uu, \ \la\overline{w(\alpha)},Y\ra >0\}$
  \Statex \qquad\qquad $\cup\; \{-\overline{w(\alpha)}, \ \alpha \in \Delta_\uu, \ \la\overline{w(\alpha)},Y\ra <0\}$
  \ForAll{$\gamma \in \Gamma_K/q\Gamma_K$}
    \State\label{algoline:Psigamma} $\Psi_{w,\gamma,\uu} \gets \{ \psi \in \Psi_{w,\uu} \setminus \Delta^+_\k, e^{\frac{2i\pi}q\la\psi,\gamma\ra}=1 \}$
    \ForAll{$\overrightarrow\sigma \in \mathcal{OS}(\Psi_{w,\gamma,\uu}, \a(\overline{w(\lambda^1)} - \mu^1))$}\label{algoline:OS}
      \State\label{algoline:S} $S^{\Sigma, w}_{\lambda,\mu}(z) \gets \Bigl( \textstyle\prod_{\beta\in \Delta_\k^+}(1-e^{-\la \beta,z\ra}) \Bigr) \frac{e^{\la{\overline{w(\lambda)}-\mu,z\ra}}}{\prod_{\alpha \in \Delta_{\uu}}(1-e^{-{\la\overline{w(\alpha)},z\ra}})}$
      \State\label{algoline:res} $g \gets g + \Res_{\overrightarrow\sigma} S^{\Sigma,w}_{\lambda,\mu}(z+\frac{2i\pi\gamma}{q})$
    \EndFor
  \EndFor
\EndFor
\State\label{algoline:return g} \Return $g$
\end{algorithmic}

\medskip

  We now recall the notation used above. We give references to the main text and comment in more detail on how we compute some of the mathematical objects involved.
  As we go along, we will also explain the correctness of our algorithm.

  Throughout, we use the conventions and identifications fixed at the beginning of Section~\ref{sec:setup}.
  Let $G=U(M)$ and $K=SU(n_2)\times\dots\times SU(n_s)$, with systems of positive roots $\Delta_\g^+$ and $\Delta_\k^+$.
  Moreover, we order the list $\Psi$ of restricted positive roots as in Remark~\ref{ex:kron restri roots}.

  In lines~\ref{algoline:lambdamu0}--\ref{algoline:lambdamu}, we think of $\lambda^0,\lambda$ as highest weights in $P\Lambda^\Sigma_{G,\geq0}$ and $\mu^0,\mu$ as highest weights in $P\Lambda_{K,\geq0}$.
  The objects $\lambda^0,\mu^0$ have integer coefficients, while $\lambda,\mu$ are expressed in terms of symbolic variables $\nu_1,\dots,\nu_s$.
  As discussed in the introduction, see Eq.~\eqref{eq:kron via branching}, the Kronecker coefficient $g(\nu^0_1,\dots,\nu^0_s)$ is equal to the branching multiplicity $m_{G,K}(\lambda^0,\mu^0)$.

  If $n_1<M-1$, then the highest weights $\lambda^0$ and $\lambda$ are singular for $G=U(M)$.
  Let $\Sigma=\{\alpha_{n_1+1},\dots,\alpha_{M-1}\}$, where $\alpha_1,\dots,\alpha_{M-1}$ denote the simple roots of $G$.
  Then $(\lambda^0,\mu^0)$ and $(\lambda,\mu)$ are in $\Lambda_{G,K,\geq 0}^\Sigma$.
  Recall that $\Delta_\uu = \Delta_\g^+ \setminus \Delta_\lp^+$, where $\Delta_\lp^+$ denotes the system of positive roots corresponding to the simple roots in $\Sigma$.
  We can identify $W_\g$ with the permutation group $S_M$ and $W_\lp$ with its subgroup $S_{\{n_1+1,\dots,M\}}$.
  Our assumptions imply that the branching cone $C_{G,K}^\Sigma$ is solid (see Example~\ref{ex:kron solid}).

  In lines~\ref{algoline:N}--\ref{algoline:epsdelta}, the sets of hyperplanes $\CA(\Psi)$ and $\CF_\Sigma$ are defined in Definitions~\ref{def:admissible etc} and \ref{def:tope for F Sigma}, respectively.
  The element $(\varepsilon,\delta)$ can be easily obtained from the inequalities of the Kirwan cone $C^\Sigma_{G,K}$, which we know in the cases that we are interested in.
  (Alternatively, it can be obtained with probability one by mapping a random unit vector in $\CH=\C^{n_1}\otimes\dots\otimes\C^{n_s}$ under the moment map, rotating it into the positive Weyl chamber, and rescaling it appropriately; cf.\ Eq.~\eqref{eq:kirwan cone}.)
  In our implementation, we precompute $(\varepsilon,\delta)$ once and reuse it in later invocations of the algorithm with the same $n_1,\dots,n_s$.

  Since each $\la \overline{w(\lambda^0)}, X \ra - \la \mu^0, X \ra$ is an integer, it is clear that $(\lambda^0,\mu^0) + t (\varepsilon,\delta)$ stays in the same tope $\tau_\Sigma$ for all $0<t\leq1$.
  In particular, $(\lambda^0,\mu^0)\in\overline \tau_\Sigma$, and if $(\lambda^0,\mu^0)\in C_{G,K}^\Sigma$ then $(\lambda^1,\mu^1)\in C_{G,K}^\Sigma$.
  (If $(\lambda^0,\mu^0)$ is already regular then the deformation is not necessary, but harmless, as it does not change the tope.)
  We may thus use Theorem~\ref{theo:branch} to compute a quasi-polynomial formula for the branching multiplicity $m_{G,K}(\lambda,\mu)$ on the closure of tope $\tau_\Sigma$.

  In lines~\ref{algoline:g0}--\ref{algoline:return g}, we implement the optimized method discussed in Remark~\ref{rem:opti}.
  We recall some of the notation involved.
  Throughout, $\overline{\psi}$ denotes the restriction of some $\psi\in i\t_\g^*$ to $i\t_\k$, which in the present case is given by the transpose of the embedding~\eqref{eq:kron embedding}.
  In lines~\ref{algoline:q} and \ref{algoline:Psigamma}, we write $\setminus$ for the difference of lists (observing multiplicities).
  The index $q$ on line~\ref{algoline:q} only depends on $n_1,\dots,n_s$.
  In our implementation, we compute it once by brute force and reuse the result for future computations.
  The element $Y$ defined in line~\ref{algoline:Y} is a regular element compatible with $\Delta_\g^+$ and $\Delta_\k^+$ (see Remark~\ref{ex:kron restri roots}).
  Thus $\Psi_{w,\u}$ in line~\ref{algoline:polari} is the set of polarized $\overline{w(\Delta_u)}$, as in the proof of Theorem~\ref{theo:branch}, and $\Psi_{w,\gamma,\u}$ the optimized subset defined in Remark~\ref{rem:opti}.
  To compute the OS bases adapted to the $\Psi_{w,\gamma,\u}$-regular element $\overline{w(\lambda^1)}-\mu^1$ on line~\ref{algoline:OS}, we use the recursive algorithm described in~\cite[Section 4.9.6]{BV2015}.
  Finally, we note that $S_{\lambda,\mu}^{\Sigma,w}$ in line~\ref{algoline:S} is a symbolic function in the variables $z$ and $\nu_1,\dots,\nu_s$ (through $\lambda,\mu$).
  Thus the residue in line~\ref{algoline:res}, and therefore the result returned by the algorithm on line~\ref{algoline:return g}, is a symbolic function in the variables $\nu_1,\dots,\nu_s$.

In view of the preceding considerations, we obtain the following theorem:

\begin{theorem}\label{theo:correct}
  The algorithm described above computes a quasi-po\-ly\-no\-mi\-al formula in the variables $\nu_1,\dots,\nu_s$ that agrees with the Kronecker coefficient $g(\nu_1,\dots,\nu_s)$ on a closed polyhedral cone which contains the point $\nu^0_1,\dots,\nu^0_s$ (namely, on the cone defined by $(\lambda,\mu)\in\overline{\tau_\Sigma}$, where $\tau_\Sigma$ is the tope described above).
\end{theorem}

\begin{remark}\label{rem:algo assumptions}
We note that the assumptions of our algorithm are without loss of generality.
The Kronecker coefficient for less than three factors is trivial to compute (e.g., $g(\nu_1,\nu_2)=\delta_{\nu_1,\nu_2}$), so we may assume that $s\geq3$.
We have $g(\nu_1,\dots,\nu_s)=0$ unless $\lvert\nu_1\rvert=\dots=\lvert\nu_s\rvert$.
Moreover, we can always remove Young diagrams with a single row, hence assume that $n_1,\dots,n_s\geq2$ and that $\nu^0_{1,2}>0$.
Lastly, we may use Lemma~\ref{lem:Cauchy reduction} to reduce to the case that $M\geq n_1\geq n_2\geq\dots\geq n_s$.
\end{remark}

\begin{remark}\label{rem:algo variants}
  There are various useful variants of our algorithm.
  If instead of using general symbolic variables $\nu_j$ we choose $\nu_j = k \nu_j^0$, where $k$ is a variable, then our algorithm computes the dilated Kronecker coefficient $k\mapsto g(k\nu_1^0,\dots,k\nu_s^0)$ as a quasi-polynomial in $k$.
  If we choose $\nu_j = \nu_j^0$ then our algorithm computes the numerical value of the Kronecker coefficient $g(\nu_1^0,\dots,\nu_s^0)$.
  As in~\cite{C-D-W}, for fixed $n_1,\dots,n_s$ both algorithms run in polynomial time (where the integers $\nu_{1,1},\dots,\nu_{s,n_s}$ are encoded in binary).

  If $\nu_1^0$ is rectangular, then the corresponding highest weight $\lambda=\tilde\nu_1^0$ is even more singular.
  Thus we may work with a larger set $\Sigma$.
  This is highly useful in practice, as it reduces the number of roots in $\Delta_\uu$ as well as the cosets of permutations $[w]$ that we need to sum over.
  Indeed, in the rectangular case, $\Sigma$ consists of all the simple roots except for one.
  When all $\nu_i$ are rectangular tableaux, this corresponds to the case of Hilbert series, and we have used this optimization for computing the results presented in Section~\ref{subsec:HS}.
  More generally, we can choose $\Sigma$ according to the highest weights that we are interested in, but care is required since the branching cone $C^\Sigma_{G,K}$ needs to be solid.
\end{remark}

A \textsc{Maple} implementation of our algorithm is available at~\cite{Kroneckerwebpage}.
All examples presented in Section~\ref{sec:examples} were computed using our algorithm on a MacBook~Pro (Intel Core~i7 processor). The running time is no more than $20$ minutes for the most difficult cases.
One exception is the example of measures of entanglement for $4$ qubits (see Example~\ref{Wallach.ex}).
This was the most challenging one to compute in terms of running time (but, on the other hand, we did not try to optimize the program for this particular case).

We have also verified our algorithm against computations made by different authors with various theoretical or computational aims (Hilbert series, stability, representations of the symmetric group, etc.).
Here is a list that is likely far from being complete:
\cite{BOR1,C-D-W,Kac,King,LUTHI,MAN1,MAN,Stem,Val,Nolan}.
In contrast to our method, some of these computations use directly the representation theory of the symmetric group (e.g., \cite{King,Stem}).
Thus they work for larger number of rows, provided that the content $\lvert\nu^0_i\rvert$ is small.
In contrast, our algorithm works best in the regime where the number of rows is fixed and not too large, but it is largely insensitive to the content and provides symbolic results.
\end{appendix}

\end{document}